\documentclass[reqno,10pt]{amsart} 

\usepackage{amssymb}
\usepackage{multirow}
\usepackage{bookmark}

\usepackage{esint}

\newtheorem{theorem}{Theorem}[section]
\newtheorem{lemma}[theorem]{Lemma}

\theoremstyle{definition}
\newtheorem{definition}[theorem]{Definition}

\theoremstyle{remark}
\newtheorem{remark}[theorem]{Remark}

\numberwithin{equation}{section}

\newcommand{\Abs}[1]{\left\lvert#1\right\rvert}
\newcommand{\norm}[1]{\lVert#1\rVert}
\newcommand{\Norm}[1]{\left\lVert#1\right\rVert}
\newcommand{\ag}[1]{\langle#1\rangle}

\newcommand{\cL}{\mathcal{L}}

\newcommand{\R}{\mathbb{R}}

\newcommand{\e}{\varepsilon}

\newcommand{\loc}{\text{loc}}

\begin{document}

\title{Uniform boundary regularity in almost-periodic homogenization}


\author{Jinping Zhuge}
\address{}
\curraddr{}
\email{}
\thanks{The author is supported in part by NSF grant DMS-1161154. The author thanks Professor Zhongwei Shen for valuable discussions and constant encouragement.}


\date{}

\begin{abstract}
	In the present paper, we generalize the theory of quantitative homogenization for second-order elliptic systems with rapidly oscillating coefficients in $APW^2(\R^d)$, which is the space of almost-periodic functions in the sense of H. Weyl. We obtain the large scale uniform boundary Lipschitz estimate, for both Dirichlet and Neumann problems in $C^{1,\alpha}$ domains. We also obtain large scale uniform boundary H\"{o}lder estimates in $C^{1,\alpha}$ domains and $L^2$ Rellich estimates in Lipschitz domains.
\end{abstract}

\maketitle
\tableofcontents
\section{Introduction}
This paper is a continuation of our previous work \cite{ShenZhuge2} and generalizes the global uniform Lipschitz estimate in periodic or uniformly almost-periodic homogenization to the second-order elliptic operators with coefficients in a broader class of discontinuous almost-periodic functions. Precisely we will study a family of elliptic operators with rapidly oscillating almost-periodic coefficients in the form of
\begin{equation}\label{def_Le}
\cL_\e = - \text{div}(A(\cdot/\e) \nabla ) = - \frac{\partial}{\partial x_i} \left\{ a^{\alpha\beta}_{ij} \bigg( \frac{\cdot}{\e}\bigg) \frac{\partial}{\partial x_j}\right\},  \qquad \e > 0
\end{equation}
where summation convention is used throughout and $\e$ is assumed to be a tiny parameter. We will assume that the coefficient matrix $A(y) = (a^{\alpha\beta}_{ij}(y))$ with $1\le i,j \le d$ and $1\le  \alpha, \beta \le m$ is real, bounded, measurable, and satisfies the following conditions:

(i) Strong ellipticity: for some $\mu>0$, and all $ y\in\R^d $ and $ \xi = (\xi_i^\alpha) \in \R^{d\times m}$,
\begin{equation}\label{def_ellipticity}
\mu|\xi|^2 \le a^{\alpha\beta}_{ij}(y) \xi_i^\alpha \xi_j^\beta \le \mu^{-1}|\xi|^2.
\end{equation}

(ii) Almost-periodicity in the sense of H. Weyl (1927): each entry of $A$ may be approximated by a sequence of trigonometric polynomials with respect to the semi-norm
\begin{equation}\label{def_W2}
\norm{f}_{W^2} = \limsup_{R\to \infty} \sup_{x\in\R^d} \left( \fint_{B(x,R)} |f|^2 \right)^{1/2}.
\end{equation}
In this situation, we also say $A\in APW^2(\R^d)$. We emphasize that this class of almost-periodic functions, which allows discontinuous functions, is much broader than that of uniformly almost-periodic functions in the sense of H. Bohr (1925) considered in \cite{Shen2, AS, AGK}, which is the closure of trigonometric polynomials with respect to the $L^\infty$ norm\cite{CC}.

We consider the following Dirichlet problem(DP) in a bounded domain $\Omega$:
\begin{equation}\label{def_DP}
\cL_\e(u_\e) +\lambda u_\e= F \quad \text{in } \Omega, \qquad \text{and} \qquad u_\e =f \quad \text{on } \partial\Omega,
\end{equation}
where $\lambda\ge 0$ is a parameter. The main goal of this paper is to establish the large scale uniform boundary Lipschitz estimates for the weak solution of (\ref{def_DP}). Here the rigorous meaning to the notion of large scale uniform boundary Lipschitz estimate is given as follows: for any $x_0\in \partial\Omega$ and any $r\ge \e$, there exists a constant $C$ independent of $\e$ or $r$, such that
\begin{equation}\label{ineq_Lip_Intro}
\left( \fint_{B(x_0,r)\cap \Omega} |\nabla u_\e|^2 \right)^{1/2} \le C.
\end{equation}

It is well-known that elliptic equations or systems (\ref{def_DP}) with discontinuous coefficients may have unbounded $\nabla u_\e$. But (\ref{ineq_Lip_Intro}) claims that $\nabla u_\e$ may be bounded in terms of average integral at a relatively large scale $r\ge \e$, uniformly with respect to $\e$, if the coefficients possess a certain repeated self-similar structure. This phenomenon also occurs in periodic homogenization and random homogenization in the stationary and ergodic setting; see \cite{Shen1, AM, AKM, ASm}. In general, (\ref{ineq_Lip_Intro}) is optimal in the sense that it does not hold uniformly for $r\ll\e$. However, as long as the assumption of local smoothness on the coefficients $A$ is imposed, a blow-up argument will send $r\to 0$ in (\ref{ineq_Lip_Intro}) and give us the usual full uniform Lipschitz estimate, i.e., $\norm{\nabla u_\e}_{L^\infty}$ is uniformly bounded; see Remark \ref{rmk_Lip}. This idea of separating large scale estimates ($r\ge \e$) only related to the homogenization process and small scale estimates ($r< \e$) only related to smoothness of coefficients has been clearly clarified in \cite{Shen1,AM} for periodic and stochastic homogenization. Therefore, in the present paper, we will focus on obtaining large scale estimate (\ref{ineq_Lip_Intro}), which reflects the essential feature of almost-periodic homogenization and meanwhile avoids the assumption of smoothness.

Let us review some background on the uniform Lipschitz estimates in homogenization before giving our main theorems. Historically, the uniform Lipschitz estimate has been studied for decades since late 1980s. The first breakthrough was due to \cite{AL} in which the authors proved the uniform Lipschitz estimates for Dirichlet problems with periodic coefficients by a compactness argument originating from the regularity theory in the calculus of variations and minimal surfaces. The compactness argument has been proved extremely useful and extensively applied in all kinds of homogenization problems; see \cite{AL2, GSh, GuS,KP,KLS2} for more references on this topic. However, the Lipschitz estimate for Neumann problems was not known until recent remarkable work in \cite{KLS2}, where the compactness argument was used along with a delicate iteration scheme. On the other hand, in \cite{ASm, AM} the authors developed a new approach in stochastic homogenization, as a replacement of compactness argument, to establish the uniform regularity estimates with a general scheme adapted to different boundary conditions. The advantage of this approach is that it relies only on the rates of convergence instead of the periodic structure or specific boundary correctors. So shortly afterwards, this general method was successfully applied in \cite{AS}, where the coefficients were assumed to be uniformly almost-periodic. In the present paper, we will use the similar approach to establish the boundary uniform Lipschitz estimates, down to scale $\e$, for operator $\cL_\e + \lambda$ with coefficients satisfying (i) and (ii).

\subsection{Main results}
To state the main results of this paper, we recall that locally the boundary of a $C^{1,\alpha}$ domain is the graph of a $C^{1,\alpha}$ function. Without loss of generality, we may consider a $C^{1,\alpha}$ function $\phi:\R^{d-1} \to \R$ with $\phi(0) = 0$ and $\norm{\nabla \phi}_{C^\alpha(\R^{d-1})} \le M$. Unless otherwise indicated, in the following main theorems and the rest of our paper, we will define
\begin{equation}\label{def_C1a}
\begin{aligned}
D_r &= \left\{  (x',x_d)\in \R^d: |x'|<r \text{ and } \phi(x') < x_d < \phi(x') + r \right\}, \\
\Delta_r &= \left\{  (x',x_d)\in \R^d: |x'|<r \text{ and } x_d = \phi(x') \right\}. \\
\end{aligned}
\end{equation}
Let $\omega_{k,\sigma}(\e) $ be the quantity defined in (\ref{def_omega}) for quantifying the rates of convergence. Then we state the main theorems of this paper as follows.

\begin{theorem}[Boundary Lipschitz estimate for DP]\label{thm_Lip_DP}
	Suppose that $A\in APW^2(\R^d)$ satisfies the ellipticity condition (\ref{def_ellipticity}) and  $\omega_{k,\sigma}$ satisfies the Dini-type condition:
	\begin{equation}\label{ineq_Dini_Intro}
	\int_0^1 \frac{\omega_{k,\sigma}(r)^{1/2}}{r} dr < \infty,
	\end{equation}
	for some $\sigma \in (0,1)$ and $k\ge 1$. Let $u_\e \in H^1(D_2;\R^d)$ be a weak solution of $\cL_\e(u_\e) + \lambda u_\e= F$ in $D_2$ with $u_\e = f$ on $\Delta_2$, where $\lambda \in [0,1]$. Then, for any $\e \le r\le 1$,
	\begin{equation}\label{ineq_DP_Lip}
	\left( \fint_{D_r} |\nabla u_\e|^2 \right)^{1/2} \le C\left\{ \left( \fint_{D_1} |\nabla u_\e|^2 \right)^{1/2} + \norm{f}_{C^{1,\tau}(\Delta_1)} + \norm{F}_{L^p(D_1)} \right\},
	\end{equation}
	where $p>d$ and $\tau\in (0,\alpha)$. The constant $C$ depends only on $A, p,\sigma,k, \tau, \alpha$ and $M$.	
\end{theorem}

We also introduce the Neumann problem(NP):
\begin{equation}\label{def_NP}
\cL_\e(u_\e) + \lambda u_\e = F \quad \text{in } \Omega, \quad \text{and} \quad \frac{\partial u_\e}{\partial \nu_\e} =g \quad \text{on } \partial\Omega, \quad \text{and} \quad \int_{\Omega} u_\e = 0,
\end{equation}
where $\lambda \ge 0$. We use $\partial u_\e/\partial\nu_\e$ to denote the co-normal derivative of $u_\e$ associated with $\cL_\e$.

\begin{theorem}[Boundary Lipschitz estimate for NP]\label{thm_Lip_NP}
	Suppose that $A\in APW^2(\R^d)$ satisfies the ellipticity condition (\ref{def_ellipticity}) and $\omega_{k,\sigma}$ satisfies the same Dini-type condition (\ref{ineq_Dini_Intro}) for some $\sigma \in (0,1), k\ge 1$. Let $u_\e \in H^1(D_2;\R^d)$ be a weak solution of $\cL_\e(u_\e) + \lambda u_\e = F$ in $D_2$ with $\partial u_\e/\partial \nu_\e = g$ on $\Delta_2$, where $\lambda \in [0,1]$. Then, for $\e \le r\le 1$,
	\begin{equation}\label{ineq_NP_Lip}
	\left( \fint_{D_r} |\nabla u_\e|^2 \right)^{1/2} \le C\left\{ \left( \fint_{D_1} |\nabla u_\e|^2 \right)^{1/2} + \norm{g}_{C^{\tau}(\Delta_1)} + \norm{F}_{L^p(D_1)} \right\},
	\end{equation}
	where $p>d$ and $\tau\in (0,\alpha)$. The constant $C$ depends only on $A, p,\sigma,k, \tau, \alpha$ and $M$.
\end{theorem}

\subsection{Strategy of proof}
We now present the outline of our approach, including some key ideas in the proof. Recall that the homogenized system is
\begin{equation}\label{def_homo}
\cL_0 u_0 + \lambda u_0 = F, \qquad \text{subject to a certain boundary condition,}
\end{equation}
where $\cL_0 = -\text{div}(\widehat{A} \nabla)$ and $\widehat{A}$ is a constant matrix known as homogenized or effective matrix which is defined in (\ref{def_Ahat}). The proof of Theorem \ref{thm_Lip_DP} or \ref{thm_Lip_NP} is roughly divided into three steps, which follow the same line as \cite{Shen1}: 

(1) Establish the $L^2$ rate of convergence in Lipschitz domains, i.e., the error estimate of $\norm{u_\e - u_0}_{L^2}$;

(2) Show that $u_\e$ satisfies the so-called \emph{flatness property} (how well it may be approximated by affine functions) as long as $u_0$ does; 

(3) Iterate step (2) down to microscopic scale $\e$, under the additional Dini-type condition (\ref{ineq_Dini_Intro}).

The rate of convergence in $L^2$ will be shown in Section 3. In fact, if $\Omega$ is a bounded Lipschitz domain, and $u_\e, u_0$ are the weak solutions of (\ref{def_DP}) and the corresponding homogenized system (\ref{def_homo}), respectively, then
\begin{equation}\label{ineq_L2Lip_Intro}
\norm{u_\e - u_0}_{L^2(\Omega)} \le C\omega_{k,\sigma}(\e)^{1/2} \Big\{ (1+\lambda)^{-1/2}\norm{F}_{L^2(\Omega)} + (1+\lambda)^{1/2} \norm{f}_{H^1(\partial\Omega)} \Big\}.
\end{equation}
The proof of (\ref{ineq_L2Lip_Intro}), in contrast to the periodic homogenization, is based on the estimates of so called approximate correctors established in \cite{ShenZhuge2}; see Section 2 for more details. We should point out that the proofs of those estimates for approximate correctors are extremely difficult and involved with compactness and ergodic arguments. The rate $O(\omega_{k,\sigma}(\e)^{1/2})$ in (\ref{ineq_L2Lip_Intro}) seems to be suboptimal. But as far as we know, it is the best result derived for almost-periodic homogenization in Lipschitz domains and it is sufficient for us to proceed with our argument for uniform Lipschitz estimates.

The second and third steps of the proof of Theorem \ref{thm_Lip_DP} and \ref{thm_Lip_NP} are laid out in section 4. Based on the \emph{flatness property}  of weak solutions of $\cL_0 + \lambda$, we are able to prove the following \emph{flatness property} of $u_\e$:
\begin{equation}\label{ineq_flatness_Intro}
H(\theta r; u_\e) \le \frac{1}{2} H(r;u_\e) + C [\omega_{k,\sigma}(\e/r)]^{1/2} \Phi(2r),
\end{equation}
for some fixed $0<\theta<1$ and all $\e < r<1$, where $H$ and $\Phi$ are defined in (\ref{def_H}) and (\ref{def_Phi}), respectively. Notice that $H(r;u_\e)$ quantifies the local regularity property of $u_\e$ and the second term on the right-hand side of (\ref{ineq_flatness_Intro}) is the error term between $u_\e$ and $u_0$. For $r>\e$, we may expect
this error term to be much smaller than the improvement in the flatness property. Then we may iterate (\ref{ineq_flatness_Intro}) down to microscopic scale $\e$ to obtain a uniform estimate for $H(r; u_\e)$ for all $\e<r<1$. This idea can be fulfilled under the extra Dini-type condition (\ref{ineq_Dini_Intro}). This is exactly the technical reason why the condition (\ref{ineq_Dini_Intro}) is necessary in our proof. Fortunately, this condition, closely related to the almost-periodicity of the coefficients $A$, can be easily satisfied in applications; see Lemma \ref{lem_rate} or Table \ref{tab_1} below.

\subsection{Further results and discussions}
In the last section, we also discuss some further applications of Lipschitz estimate and its proof. The first application is to improve the estimate for $\nabla \chi_T$ by interior Lipschitz estimate for $\cL_\e + \lambda$ with $\lambda = 1$. We show that with (\ref{ineq_Dini_Intro}) imposed, $\norm{\nabla \chi_T}_{S_1^2} $ is uniformly bounded, instead of just being bounded by $CT^\sigma$ for $\sigma>0$. The second application is devoted to the large scale Rellich estimate in $L^2$. More precisely, we will show that
\begin{equation}\label{ineq_RelDP}
\left(\fint_{\Omega_r} |\nabla u_\e|^2 \right)^{1/2} \le C\norm{\nabla_{\text{tan}} u_\e}_{L^2(\Omega)},
\end{equation}
and
\begin{equation}\label{ineq_RelNP}
\left(\fint_{\Omega_r} |\nabla u_\e|^2 \right)^{1/2} \le C\norm{\frac{\partial u_\e}{\partial \nu_\e} }_{L^2(\Omega)},
\end{equation}
for all $r\ge\omega_{k,\sigma}(\e)$. These estimates imply the usual Rellich estimate if $\omega_{k,\sigma}(\e) = O(\e)$ and $A$ possesses symmetry and certain smoothness; see the remark after Theorem \ref{thm_Rellich_NP}.

The last application is the large scale boundary H\"{o}lder estimates for both Dirichlet and Neumann problems. The argument follows a similar but simpler way as boundary Lipschitz estimate. The main point here is that we do not need any extra condition on the convergence rate $\omega_{k,\sigma}(\e)$. Indeed, the fact $\omega_{1,\sigma}(\e) \to 0$ as $\e\to 0$ is sufficient for us to establish the uniform boundary H\"{o}lder estimate. We state the result as follows.

\begin{theorem}[Boundary H\"{o}lder estimate for DP]\label{thm_holder_DP}
	Suppose that $A\in APW^2(\R^d)$ satisfies the ellipticity condition (\ref{def_ellipticity}). Let $u_\e \in H^1(D_2;\R^d)$ be a weak solution of $\cL_\e(u_\e) + \lambda u_\e= F$ in $D_2$ with $u_\e = f$ on $\Delta_2$, where $\lambda \in [0,1]$. Then, for any $\e \le r\le 1$,
	\begin{equation}\label{ineq_Holder_Intro}
	\left( \fint_{D_r} |\nabla u_\e|^2 \right)^{1/2} \le  C r^{\gamma-1} \Bigg\{ \left( \fint_{D_1} |\nabla u_\e |^2 \right)^{1/2} + \left( \fint_{D_1} |F|^p \right)^{1/p}
	+ \norm{f}_{C^{0,1}(\Delta_{1})} \Bigg\},
	\end{equation}
	where $\gamma < 2-d/p, p\ge 2, p>d/2$. In particular, if $p = d$, then (\ref{ineq_Holder_Intro}) holds for all $\gamma \in (0,1)$.
\end{theorem}

A similar estimate also works for Neumann problems (\ref{def_NP}) with $\norm{f}_{C^{0,1}(\Delta_{1})}$ replaced by $\norm{g}_{L^\infty(\Delta_{1})}$; see Theorem \ref{thm_holder_NP}.

Overall, we can see from previous results the close relationship between the rate $\omega_{k,\sigma}(\e)$ and uniform regularity in different situations. This idea more or less has been shown in \cite{Shen1} for periodic homogenization, whose rate of convergence is always the same, i.e., $\omega_{1,\sigma}(\e) = O(\e)$. But it is of particular interest for almost-periodic homogenization since the rate of convergence could be arbitrarily slow. In the following table, we will summarize all the uniform regularity results obtained in this paper and \cite{ShenZhuge2}, and clarify how the function $\omega_{k,\sigma}(\e)$, which quantifies the rate of convergence, is related to the certain uniform regularity.

\begin{table}[h]
	\caption{Relationship between convergence rate and regularity}\label{tab_1}
	\begin{center}
		\begin{tabular}{|l|l|l|}
			\hline
			\bf Sufficient condition on $A$ & \bf Rate of convergence & \bf Large scale regularity\\
			\hline
			No extra condition needed & $\omega_{1,\sigma}(\e) \to 0$ as $\e \to 0 $ & H\"{o}lder estimates  \\
			\hline
			\multirow{2}{*}{$\rho_k(L,L) \lesssim \ln(1+L)^{-\alpha}, \alpha > 3$} & \multirow{2}{*}{$\displaystyle \int_0^1 \frac{\omega_{k,\sigma}(r)^{1/2}}{r} dr< \infty $ }& Lipschitz estimates;\\
			& & and $\norm{\nabla \chi_T}_{S_1^2} \le C$ \\
			\hline
			$\rho_k(L,L) \lesssim L^{-1-\alpha}, \alpha > 0$; & \multirow{3}{*}{$\omega_{k,\sigma}(\e ) \lesssim \e$ }& $L^2$ Rellich estimate; \\
			or $A$ is sufficiently smooth and& & existence of true corrector  \\
			quasi-periodic\cite{AGK} & & $\chi$ and $\norm{\chi}_{S_1^2} \le C$\cite{ShenZhuge2} \\
			\hline
		\end{tabular}
	\end{center}
\end{table}

Throughout this paper we will use $\fint_E f$ to denote the average integral of function $f$ over a set $E$, and $C$ to denote constants that depend at most on $A,\Omega$ 
and other scale-independent parameters(e.g. $k, \sigma, p$, etc.), but never on $\e, T$ or other scale-dependent parameters(e.g., $\lambda, L,r$, etc.).

\section{Preliminaries for almost-periodic homogenization}
In this section, we will briefly review some preliminaries of almost-periodic homogenization along with particular emphasis on the characterization of almost-periodicity and approximate correctors. Except for some classical contents, most of the them were formulated in our recent paper \cite{ShenZhuge2}.

\subsection{Homogenization}
We start with spaces of almost-periodic functions. Let $\text{Trig}(\R^d)$ denote the set of real trigonometric polynomials in $\R^d$. A function $f$ in $L^2_{\text{loc}}(\R^d) $ is said to belong to $B^2(\R^d)$ if $f$ is the limit of a sequence of functions in $\text{Trig}(\R^d)$ with respect to the semi-norm
\begin{equation}\label{cond_B2}
\norm{f}_{B^2} = \limsup_{R\to\infty} \left( \fint_{B(0,R)} |f|^2 \right)^{1/2}.
\end{equation}
Functions in $B^2(\R^d)$ are usually said to be almost-peiodic in the sense of Besicovitch (1926). It is not hard to see that if $g\in L^\infty(\R^d) \cap B^2(\R^d)$ and $f\in B^2(\R^d)$, then $fg\in B^2(\R^d)$.

Let $f\in L^1_{\text{loc}}(\R^d)$. A number $\ag{f}$ is called the mean value of $f$ if
\begin{equation}
\lim_{\e\to 0^+} \int_{\R^d} f(x/\e) \varphi(x) dx = \ag{f} \int_{\R^d} \varphi,
\end{equation}
for any $\varphi \in C_0^\infty(\R^d)$. It is known that if $f\in B^2(\R^d)$, then $f$ has a mean value. Under the equivalent relation that $f\sim g$ if $\norm{f-g}_{B^2} = 0$, the set $B^2(\R^d)$ is a Hilbert space with the inner product defined by $(f,g) = \ag{fg}$. Furthermore, one has the following Weyl's orthogonal decomposition,
\begin{equation}
B^2(\R^d;\R^{d\times m}) = V^2_{\text{pot}} \oplus V^2_{\text{sol}} \oplus \R^{d\times m}.
\end{equation}
where $V^2_{\text{pot}}$ (resp., $V^2_{\text{sol}}$) denotes the closure of potential (resp., solenoidal) trigonometric polynomials with mean value zero in $B^2(\R^d;\R^{d\times m})$. Assume $A =(a^{\alpha\beta}_{ij}) \in B^2(\R^d)$ satisfies the ellipticity condition (\ref{def_ellipticity}). For each $1\le j\le d$ and $1\le \beta \le m$, let $\psi_j^\beta = (\psi_{ij}^{\alpha\beta})$ be the unique function in $V^2_{\text{pot}}$ satisfying the following auxiliary equations
\begin{equation}\label{def_Vpot}
(a_{ik}^{\alpha\gamma} \psi_{kj}^{\gamma\beta}, \phi_i^\alpha) = - (a_{ij}^{\alpha\beta}, \phi_i^\alpha) \qquad \text{for any } \phi = (\phi_i^\beta) \in V^2_{\text{pot}}.
\end{equation}
It is shown in \cite{JKO} that $A$ admits homogenization with homogenized matrix $\widehat{A} = (\widehat{a}_{ij}^{\alpha\beta})$ defined by
\begin{equation}\label{def_Ahat}
\widehat{a}_{ij}^{\alpha\beta} = \ag{a_{ij}^{\alpha\beta}} + \ag{a_{ik}^{\alpha\gamma} \psi_{kj}^{\gamma\beta}}.
\end{equation}
Moreover, $\widehat{A^*} = (\widehat{A})^*$, where $A^*$ denotes the adjoint of $A$. The following is a statement of the homogenization theorem whose proof actually was contained in \cite{ShenZhuge2}.

\begin{theorem}\label{thm_homo}
	Suppose $A=(a_{ij}^{\alpha\beta})$ satisfies the ellipticity condition  (\ref{def_ellipticity}) and each $a_{ij}^{\alpha\beta} \in B^2(\R^d)$. Let $\Omega$ be a bounded Lipschitz domain in $\R^d$ and $F\in H^{-1}(\Omega;\R^m)$. Let $u_\e \in H^1(\Omega;\R^m)$ be a weak solution of $\cL_\e(u_\e) + \lambda u_\e = F$. Suppose that $u_\e \to u_0$ weakly in $H^1(\Omega;\R^m)$. Then $A(x/\e) \nabla u_\e \to \widehat{A} \nabla u_0$ weakly in $L^2(\Omega;\R^{d\times m})$. Consequently, if $f\in H^{1/2}(\partial \Omega;\R^m)$ and $u_\e$ is the weak solution to the Dirichlet problem:
	\begin{equation}
	\cL_\e(u_\e) + \lambda u_\e= F \quad \text{ in } \Omega \quad \text{and} \quad u_\e=f \quad \text{on } \partial\Omega,
	\end{equation}	
	Then, as $\e\to 0$, $u_\e\to u_0$ weakly in $H^1 (\Omega;\R^m)$, where $u_0$ is the weak solution to
	\begin{equation}
	\cL_0(u_0) + \lambda u_0= F \quad \text{ in } \Omega \quad \text{and} \quad u_0=f \quad \text{on } \partial\Omega.
	\end{equation}  
\end{theorem}
We point out here that $B^2(\R^d)$ is usually the largest space of almost-periodic functions in which the homogenization theorem could be established. However, this space seems unsuitable for obtaining quantitative theory due to the lack of spacial uniformity.

\subsection{Almost-periodicity and approximate correctors}
We define a subspace of $B^2(\R^d)$,
\begin{equation*}
APW^2(\R^d) = \text{the closure of Trig} (\R^d) \text{ with respect to } W^2 \text{ semi-norm},
\end{equation*}
where the $W^2$ semi-norm is defined in (\ref{def_W2}). The functions in $APW^2(\R^d)$ are called almost-periodic in the sense of H. Weyl. Note that in the definition of $APW^2(\R^d)$, the regularity assumption is completely removed and hence this space is much larger than 
the classes of uniformly almost-periodic functions considered in \cite{Shen2, AS, AGK}. Earlier work in \cite{ShenZhuge2} also indicates that $APW^2(\R^d)$ is a fairly suitable space for quantitative homogenization. All of the following settings and results concerning the coefficient matrix $A\in APW^2(\R^d)$ were formulated in \cite{ShenZhuge2}, based on the ideas of \cite{Shen2} and \cite{AGK}.

For $g\in L^p_{\loc}(\R^d)$ and $R>0$, we define  the  norm
\begin{equation}\label{def_Sp}
\| g\|_{S_R^p} =\sup_{x\in \R^d} \left(\fint_{B(x, R)} |g|^p\right)^{1/p}.
\end{equation}
For $y, z\in \R^d$, define the difference operator
\begin{equation}\label{def_Diff}
\Delta_{yz}  g (x):=  g(x+y)- g(x+z) .
\end{equation}
Let $P=P_k =\big\{ (y_1, z_1), \dots, (y_k, z_k) \big\}$ be a collection of pairs $(y_i, z_i)\in \R^d\times \R^d$, and
\begin{equation*}
Q=\big\{ (y_{i_1}, z_{i_1}), \dots, (y_{i_\ell}, z_{i_\ell}) \big\}
\end{equation*}
be a subset of $P$ with $i_1<i_2<\dots <i_\ell$. Define
\begin{equation*}
\Delta_Q (g)=\Delta_{y_{i_1} z_{i_1}} \cdots \Delta_{y_{i_\ell} z_{i_\ell}} (g).
\end{equation*}
To quantify the almost periodicity of the coefficient matrix $A$,
we introduce 
\begin{equation}\label{def_rho_k}
\rho_{k} ( L, R)
=\sup_{y_1\in \R^d}\inf_{|z_1|\le L} \cdots\sup_{y_k\in \R^d} \inf_{|z_k|\le L}
\sum
\|\Delta_{Q_1} (A)\|_{S^p_R} 
\cdots \|\Delta_{Q_\ell}  (A)\|_{S^p_R},
\end{equation}
where the sum is taken over all partitions of $P=Q_1\cup Q_2 \cup \cdots \cup Q_\ell$ with $1\le \ell\le k$ and $Q_i\cap Q_j = \emptyset$ if $i\neq j$.
The exponent $p$ in  (\ref{def_rho_k}) depends on $k$ and is given by
\begin{equation}\label{def_expt_p}
\frac{k}{p} =\frac{1}{2}-\frac{1}{\bar{q}},
\end{equation}
where $\bar{q}>2$ is the exponent related to the reverse H\"older estimate (Meyers' estimate) of solutions of elliptic operators, which depends only on $d,m$ and $\mu$. Note that $\rho_k(L,R) \le C_k \rho_1(L,R)$ and $\rho_1(L,R) \to 0$ as $L,R\to\infty$.

\begin{definition}
	Let $P_j^\beta(x) = x_j e^\beta$, where $e^\beta = (0,\cdots,1,\cdots,0)$ with $1$ in the $\beta^{\text{th}}$ position. For any $T>0$, let $u = \chi^\beta_{T,j} = (\chi^{1\beta}_{T,j},\cdots,\chi^{m\beta}_{T,j})$ be the weak solution of
	\begin{equation}\label{def_corrector}
	-\text{div} (A(x) \nabla u) + T^{-2} u =  \text{div}(A(x)\nabla P_j^\beta) \quad \text{in } \R^d,
	\end{equation}
	given by \cite[Lemma 3.1]{ShenZhuge2}. The matrix-valued functions $\chi_T = (\chi_{T,j}^\beta) = (\chi_{T,j}^{\alpha\beta})$ are called the approximate correctors for the family of operators $\{\cL_\e \}$.
\end{definition}

The importance of approximate correctors is due to the fact that
\begin{equation}\label{ineq_chiT2psi}
\norm{\nabla \chi_T - \psi}_{B^2} \to 0, \qquad \text{as } T \to \infty,
\end{equation}
and thus $\chi_T$ could be regarded as an approximation of the usual correctors.

\begin{theorem}\label{thm_chiT_1}
	Suppose that $A\in APW^2(\R^d)$ and satisfies the ellipticity condition (\ref{def_ellipticity}). 
	Fix $k\ge 1$ and $\sigma\in (0,1)$.
	Then there exists a constant $c>0$, depending only on $d$ and $k$, such that
	for any $T\ge 2$, 
	\begin{equation}\label{ineq_dchiT_1}
	\|\nabla \chi_T\|_{S^2_1} \le C_\sigma T^\sigma,
	\end{equation}
	and
	\begin{equation}\label{ineq_chiT_1}
	\|\chi_T\|_{S^2_1}
	\le C_\sigma \Theta_{k,\sigma}(T),
	\end{equation}
	where $C_\sigma$ depends only on $\sigma, k$ and $A$, and $\Theta_{k,\sigma}$ is defined by
	\begin{equation}\label{def_Theta}
	\Theta_{k,\sigma}(T) = \int_1^T \inf_{1\le L\le t}
	\left\{ \rho_k (L, t) +\exp\left(-\frac{c\, t^2}{L^2} \right) \right\}
	\left(\frac{T}{t}\right)^\sigma dt.
	\end{equation}
\end{theorem}

\begin{definition}
	For $1\le i,j\le d$ and $1\le \alpha,\beta \le m$, define
	\begin{equation}\label{def_bT}
	b^{\alpha\beta}_{T,ij}(y) = a^{\alpha\beta}_{ij}(y) + a^{\alpha\gamma}_{ik} \frac{\partial}{\partial y_k} \big( \chi^{\gamma\beta}_{T,j}(y)\big) - \widehat{a}^{\alpha\beta}_{ij}.
	\end{equation}
	We call $\phi^{\alpha\beta}_{T,ij} \in H^2_{\text{loc}}(\R^d)$ the dual approximate correctors if they are the solutions of
	\begin{equation}\label{def_phiT}
	-\Delta \phi^{\alpha\beta}_{T,ij} + T^{-2} \phi^{\alpha\beta}_{T,ij} = b^{\alpha\beta}_{T,ij} - \ag{b^{\alpha\beta}_{T,ij}},
	\end{equation}
	given by \cite[Lemma 3.1]{ShenZhuge2}.
\end{definition}

\begin{theorem}\label{thm_dualcor}
	Suppose that $A\in APW^2(\R^d)$ satisfies the ellipticity condition (\ref{def_ellipticity}). Then, for any $\sigma\in (0,1), k\ge 1$ and $T\ge 2$,
	\begin{align}\label{ineq_phiT_B1}
	\begin{aligned}
	\norm{T^{-1} \phi_T}_{S_1^2} + \norm{\nabla \phi_T}_{S_1^2} + \norm{ T \nabla \frac{ \partial }{\partial x_i} \phi_{T,i\cdot}}_{S_1^2} \le C_\sigma \Theta_{k,\sigma}(T),
	\end{aligned}
	\end{align}
	where the constant $C_\sigma$ depends only on $\sigma,k$ and $A$.
\end{theorem}

Also throughout this paper we define
\begin{equation}\label{def_omega}
\omega_{k,\sigma}(\e) = \norm{\psi-\nabla \chi_T}_{B^2} + \norm{\psi^*-\nabla \chi_T^*}_{B^2} + T^{-1}\Theta_{k,\sigma}(T), \qquad \e = T^{-1},
\end{equation}
where $\psi^*$ and $\chi_T^*$ are the auxiliary functions and approximate correctors for the family of operators $\{\cL_\e^* \}$. We are willing to emphasize that the quantity $\omega_{k,\sigma}(\e)$ plays an important role in this paper since it perfectly characterizes the almost-periodicity of coefficients $A$, in the sense of H. Weyl, and quantifies the rate of convergence; see Theorem \ref{thm_L2_C11} and \ref{thm_H1_Lip}. Actually, it is shown in \cite{ShenZhuge2} that if $A$ satisfies the ellipticity condition (\ref{def_ellipticity}) and $A\in APW^2(\R^d)$, then $\omega_{k,\sigma}(\e) \to 0$ as $\e \to 0$. In particular, the following lemma gives the explicit control on $\omega_{k,\sigma}(\e)$ in terms of $\rho_k$, which quantifies the almost-periodicity of $A$.
\begin{lemma}\label{lem_rate}
	Suppose that $A\in APW^2(\R^d)$ satisfies the ellipticity condition (\ref{def_ellipticity}). Then
	\begin{equation}\label{ineq_omg}
	\omega_{k,\sigma}(\e) \le C\int_{T}^{\infty} \frac{\Theta_{k,\sigma}(t)}{t^2} dt + CT^{-1}\Theta_{k,\sigma}(T),
	\end{equation}
	provided the integral on the right-side is bounded, where $T = \e^{-1}$. In particular,
	
	(i) if $\rho_k(L,L) \le C L^{-1-\alpha}$ for some $\alpha>0,k\ge 1$, then
	\begin{equation}
	\omega_{k,\sigma'}(\e) = O(\e),
	\end{equation}  for some $\sigma'$ that depends only on $\alpha$;
	
	(ii) if $\rho_k(L,L) \le C L^{-\alpha}$ for some $0<\alpha \le 1,k\ge 1$, then
	\begin{equation}
	\omega_{k,\sigma'}(\e) = O(\e^\beta),
	\end{equation}  for all $\beta<\alpha$ and some $\sigma'$ that depends only on $\alpha,\beta$;
	
	(iii) if $\rho_k(L,L) \le C \ln(1+L)^{-\alpha}$ for some $\alpha > 1, k\ge 1$, then
	\begin{equation}
	\omega_{k,\sigma}(\e) = O((-\ln \e)^{1-\alpha}).
	\end{equation}
\end{lemma}

\begin{proof}
	In view of (\ref{def_omega}), to see (\ref{ineq_omg}), it suffices to show
	\begin{equation}
	\norm{\psi-\nabla \chi_T}_{B^2} \le C\int_{T}^{\infty} \frac{\Theta_{k,\sigma}(t)}{t^2} dt.
	\end{equation}
	Observe that (\ref{ineq_chiT2psi}) implies
	\begin{equation*}
	\norm{\psi-\nabla \chi_T}_{B^2} \le \sum_{j=1}^\infty \norm{\nabla \chi_{2^{j-1} T}-\nabla \chi_{2^{j} T}}_{B^2}.
	\end{equation*}
	Let $v = \chi_{2^{j-1} T}- \chi_{2^{j} T}$. It follows from the definition of $\chi_T$ that $v$ satisfies
	\begin{equation*}
	-\text{div} (A \nabla v) + (2^{j-1}T)^{-2} v = -3 (2^{j}T)^{-2} \chi_{2^j T}.
	\end{equation*}
	Therefore, a standard estimate in \cite[Lemma 3.1]{ShenZhuge2} claims that
	\begin{equation*}
	\norm{\nabla v}_{S_{2^{j-1}T}^2} = \norm{\nabla \chi_{2^j T}-\nabla \chi_{2^{j-1} T}}_{S_{2^{j-1} T}^2} \le C (2^{j}T)^{-1}\norm{\chi_{2^jT}}_{S_{2^{j-1} T}^2}.
	\end{equation*}
	Hence, it follows from (\ref{ineq_chiT_1}) that
	\begin{align*}
	\norm{\psi-\nabla \chi_T}_{B^2} &\le C\sum_{j=1}^\infty (2^j T)^{-1} \norm{\chi_{2^jT}}_{S_{2^{j-1} T}^2} \\
	&\le C\sum_{j=1}^\infty (2^j T)^{-1} \Theta_{k,\sigma}(2^j T) \\
	&\le C \sum_{j=1}^\infty \int_{2^{j-1}T}^{2^j T} \frac{\Theta_{k,\sigma}(t)}{t^2} dt \\
	&= C\int_{T}^{\infty} \frac{\Theta_{k,\sigma}(t)}{t^2} dt.
	\end{align*}
	The estimate for $\norm{\psi^*-\nabla \chi_T^*}_{B^2}$ is exactly the same.
	
	Now we recall that (i) was actually shown in \cite{ShenZhuge2} and we skip the proof here. Parts (ii) and (iii) are direct corollaries of (\ref{ineq_omg}). Indeed, for (ii), if $\rho_k(L,L) \le C L^{-\alpha}$ for some $0<\alpha \le 1$, then
	\begin{align*}
	\inf_{1\le L\le t}
	\left\{ \rho_k (L, t) +\exp\left(-\frac{c\, t^2}{L^2} \right) \right\} &\le \rho_k (t^{\delta}, t) +\exp\left(-t^{2(1-\delta)} \right) \\
	& \le C t^{-\delta \alpha}
	\end{align*}
	for any $0<\delta<1$ and $C$ depends also on $\delta$ and $\alpha$. Choosing $\sigma$ appropriately such that $\delta \alpha + \sigma\neq 1$, we obtain
	\begin{equation}
	\Theta_{k,\sigma}(T) \le C T^\sigma \int_1^T t^{-\delta \alpha - \sigma} dt \le C T^{1-\delta\alpha}.
	\end{equation}
	As a result,
	\begin{equation}
	\omega_{k,\sigma}(\e) \le C \int_T^\infty \frac{t^{1-\delta \alpha}}{t^2} dt + C T^{-1} T^{1-\delta \alpha} \le C T^{-\delta\alpha} = C\e^\beta,
	\end{equation}
	where $\beta = \delta\alpha$ could be any number less than $\alpha$ since $0<\delta<1$ is arbitrary. The estimate for (iii) is similar.
\end{proof}

\begin{remark}
	we point out here that part (ii) in the lemma above will not be used in this paper. Part (i) will be involved in the Rellich estimate in the last section. And part (iii) provides a sufficient condition on the coefficients for the Dini-type condition (\ref{ineq_Dini_Intro}). Actually, it is not hard to see that if $\rho_k(L,L) \le C \ln(1+L)^{-\alpha}$ for some $k\ge 1$ and $\alpha > 3$, then (\ref{ineq_Dini_Intro}) is satisfied.
	Also, we should mention that these estimates of decay for $\rho_k$ hold for any periodic coefficients or sufficiently smooth quasi-periodic coefficients (see \cite{AGK} for example). This means that all the work in this paper generalizes the results in periodic homogenization and is really applicable to non-periodic homogenization.
\end{remark}

\subsection{A framework for convergence rates}
In this subsection, we will introduce a framework for obtaining rates of convergence in $L^2$ space. This framework was formulated in \cite{ShenZhuge1} for mixed boundary value problems with periodic coefficients, which was motivated by earlier work in \cite{GG,Su1,Su2}. The advantage of this framework is that we can handle homogenization problems with different boundary conditions and non-periodic coefficients in a more efficient uniform fashion. To see this, we first introduce some notations and lemmas. Let $\zeta \in C_0^\infty(B_1(0))$ be a cut off function with $\int \zeta = 1$ and $\zeta_\e(x) = \e^{-d} \zeta(x/\e) $. Define the smoothing operator
\begin{equation}
S_\e f(x) = \zeta_\e* f(x) = \int_{\R^d} \zeta_\e(y) f(x-y) dy.
\end{equation}
Clearly, for $1\le p\le \infty$,
\begin{equation}\label{ineq_S_p}
\norm{S_{\e} f}_{L^p(\Omega)} \le \norm{f}_{L^p(\Omega)}.
\end{equation}

Let $\delta > 2\e$ be a small parameter to be determined. Let $\eta_\delta \in C_0^\infty(\Omega)$ be a cut-off function so that $\eta_\delta(x) =0 $ in $\Omega_\delta = \{x\in\Omega; \text{dist}(x,\partial\Omega) < \delta\}$, $\eta_\delta(x) = 1$ in $\Omega \setminus \Omega_{2\delta}$ and $|\nabla \eta_\delta(x)| \le C\delta^{-1}$. Then define the so-called localized smoothing operator as
\begin{equation}
K_{\e,\delta} f(x) = S_\e (\eta_\delta f)(x).
\end{equation}
Note that $K_{\e,\delta} f \in C_0^\infty(\Omega)$ since $\delta > 2\e$.

Lemma \ref{lem_gxe_B1}-\ref{lem_Oeu_g} are standard and their proofs may also be found in \cite{ShenZhuge1}.
\begin{lemma}\label{lem_gxe_B1}
	Let $f\in L^p(\R^d)$ for some $1\le p< \infty$. Then for any $g\in L^p_{\text{loc}}(\R^d)$,
	\begin{equation}
	\norm{g(\cdot/\e)S_\e(f)}_{L^p(\R^d)}\le \norm{g}_{S_1^p} \norm{f}_{L^p(\R^d)},
	\end{equation}
	where $C$ depends only on $d$.
\end{lemma}

\begin{lemma}\label{lem_S_e}
	Let $f\in W^{1,p}(\R^d)$ for some $1\le p<\infty$. Then
	\begin{equation*}
	\norm{S_\e(f) - f}_{L^p(\R^d)} \le C \e \norm{\nabla f}_{L^p(\R^d)},
	\end{equation*}
	where $C$ depends only on $d$.
\end{lemma}

\begin{lemma}\label{lem_Oeu_g}
	Let $\Omega$ be a bounded Lipschitz domain, then for any $u\in H^1(\R^d)$ and any $r\ge \e$,
	\begin{equation*}
	\int_{\Omega_{r}} |g(\cdot /\e) S_\e(u)|^2 \le C r \norm{g}_{S_1^2}^2 \norm{u}_{H^1(\R^d)} \norm{u}_{L^2(\R^d)},
	\end{equation*}
	where $\Omega_r = \{x\in \Omega: \text{dist}(x,\partial\Omega) < r\}$ and the constant $C$ depends only on the domain $\Omega$.	
\end{lemma}

As we know, the $H^1$ estimate, i.e., the error estimate of the first order approximation, is usually the first step to establish the $L^2$ estimates of $u_\e - u_0$. We introduce our modified first order approximation, which is defined as follows:
\begin{equation}\label{def_1st-appr}
w_\e = u_\e - u_0 -\e \chi_{T,k}^\beta(x/\e) K_{\e,\delta} \left(\frac{\partial u_0^\beta}{\partial x_k}\right),
\end{equation}
where $u_\e, u_0$ are the weak solutions associated with $\cL_\e$ and $\cL_0$, respectively. The operator $K_{\e,\delta}$ in the correction term has two effects: (i) thanks to Lemma \ref{lem_gxe_B1}, the uniform boundedness of approximate correctors $\chi_T$ is not necessary for $L^2$ estimate of the correction term; (ii) the presence of the cut-off function avoids extra effect of boundary correctors or boundary regularity.

Now we state the theorem of $L^2$ convergence rate in $C^{1,1}$ domains as follows:
\begin{theorem}\label{thm_L2_C11}
	Suppose that $\Omega$ is a bounded $C^{1,1}$ domain and $A\in APW^2(\R^d)$ satisfies the ellipticity condition (\ref{def_ellipticity}). Let $u_\e$ be the weak solution  of (\ref{def_DP}) or (\ref{def_NP}) and $u_0\in H^2(\Omega)$ be the weak solution of homogenized system with the same boundary data. Then
	\begin{equation}\label{ineq_L2DN_C11}
	\Norm{u_\e - u_0}_{L^2(\Omega)} \le C\omega_{k,\sigma}(\e) \Big\{ \norm{\nabla^2 u_0}_{L^2(\Omega)} + (1+\lambda \omega_{k,\sigma}(\e))^{1/2} \norm{\nabla u_0}_{L^2(\Omega)} \Big\},
	\end{equation}
	where $\omega_{k,\sigma}(\e)$ is defined in (\ref{def_omega}) and $C$ depends only on $\sigma,k, A$ and Lipschitz character of $\Omega$.
\end{theorem}

This theorem was essentially proved in \cite{ShenZhuge2} with $\lambda = 0$ and Dirichlet boundary condition. The cases with positive $\lambda$ or Neumann boundary condition follow from a similar argument. The novelty in the theorem we stated above is that we figure out explicitly how the bound depends on $\lambda$ and the certain derivatives. The sketch of the proof is as follows. The main step is to show that for any test function $\varphi\in H^1(\Omega)$,
\begin{equation}\label{ineq_dwdphi}
\begin{aligned}
&\int_{\Omega} A(x/\e) \nabla w_\e \cdot \nabla \varphi + \lambda \int_{\Omega} w_\e \cdot \varphi \\
&\qquad \le C\lambda \omega_{k,\sigma}(\e) \norm{\varphi}_{L^2(\Omega)} \norm{\nabla u_0}_{L^2(\Omega)} \\
&\qquad \qquad + C\Big\{ \omega_{k,\sigma}(\e) \norm{\nabla \varphi}_{L^2(\Omega)}+ \omega_{k,\sigma}(\e)^{1/2} \norm{\nabla \varphi}_{L^2(\Omega_{4\delta})} \Big\} \norm{\nabla u_0}_{H^1(\Omega)}.
\end{aligned}
\end{equation}
The proof of this inequality is vary similar to \cite[Lemma 10.4]{ShenZhuge2}, which will be skipped here. Observe that (\ref{ineq_dwdphi}) gives exactly the $H^1$ convergence rate if we set $\varphi = w_\e$ and bound $\norm{\nabla \varphi}_{L^2(\Omega_{4\delta})}$ roughly by $\norm{\nabla \varphi}_{L^2(\Omega)}$, i.e.
\begin{equation}\label{ineq_H1DN_C11}
\begin{aligned}
&\norm{ w_\e}_{H^1(\Omega)} + (1+\lambda)^{1/2} \norm{w_\e}_{L^2(\Omega)} \\
&\qquad \qquad \le C\omega_{k,\sigma}(\e)^{1/2} \Big\{ \norm{\nabla^2 u_0}_{L^2(\Omega)} + C(1+\lambda \omega_{k,\sigma}(\e))^{1/2} \norm{\nabla u_0}_{L^2(\Omega)} \Big\},
\end{aligned}
\end{equation}
Then we can use a duality argument, combining with (\ref{ineq_dwdphi}) and (\ref{ineq_H1DN_C11}), to improve the $L^2$ rate of convergence. The duality argument, as an indispensable part of our framework, has been used in \cite{ShenZhuge1}. Finally, we should mention that for periodic case, the result of Theorem \ref{thm_L2_C11} has also been proved in \cite{Su1} and \cite{Su2}, without showing how the constant depends on $\lambda$.

\section{Convergence rates in Lipschitz domains}
Recently, the sharp rates of convergence in $C^{1,1}$ domains were obtained for variational elliptic problems with rough periodic coefficients; see \cite{GG,Su1,Su2,ShenZhuge1,Gu} for example.  Theorem \ref{thm_L2_C11} possibly gives the nearly sharp rate of convergence for elliptic systems with almost-periodic coefficients in H. Weyl's sense. These results are restricted to $C^{1,1}$ domains and $u_0$ has to be in $H^2(\Omega)$, which are sufficient for the interior Lipschitz estimate as in \cite{ShenZhuge2}. However, they are definitely insufficient for boundary Lipschitz estimate with $C^{1,\alpha}$ domains and boundary data considered in this paper. In this section, we will extend the rate of convergence from $C^{1,1}$ domains to general Lipschitz domains for elliptic systems with almost-periodic coefficients. Our argument follows the same ideas as \cite{Shen1} and particularly relies on the solvability of $L^2$ elliptic boundary value problems with constant coefficients in Lipschitz domains. Now we state the main result of this section as follows:

\begin{theorem}\label{thm_H1_Lip}
	Suppose that $\Omega$ is a bounded Lipschitz domain and $A\in APW^2(\R^d)$ satisfies the ellipticity condition (\ref{def_ellipticity}). Let $u_\e$ be the weak solution  of (\ref{def_NP}) and $u_0$ be the weak solution of homogenized system (\ref{def_homo}) with the same boundary data. Let $w_\e$ be the first order approximation defined in (\ref{def_1st-appr}),
	then
	\begin{equation}\label{ineq_H1_Lip}
	\norm{ w_\e}_{H^1(\Omega)} + (1+\lambda)^{1/2} \norm{w_\e}_{L^2(\Omega)} \le C\omega_{k,\sigma}(\e)^{1/2} \Big\{\norm{F}_{L^2(\Omega)} + (1+\lambda) \norm{f}_{H^1(\partial\Omega)}\Big\},
	\end{equation}
	where $T = \e^{-1}, \delta = T^{-1}\Theta_{k,\sigma}(T)$ and $C$ depends only on $\sigma, k, A$ and Lipschitz character of $\Omega$.
\end{theorem}

Observe that (\ref{ineq_H1_Lip}) is exactly the generalization of (\ref{ineq_H1DN_C11}) in Lipschitz domains if we take the energy estimate for $u_0$ into account. However, the proof for (\ref{ineq_H1_Lip}) will be more involved since $\nabla^2 u_0$ may not be an $L^2$ function in general. Also note that, as a corollary, Theorem \ref{thm_H1_Lip} provides an $L^2$ rate of convergence
\begin{equation}\label{ineq_L2_Lip}
\norm{u_\e - u_0}_{L^2(\Omega)} \le C\omega_{k,\sigma}(\e)^{1/2} \Big\{(1+\lambda)^{-1/2}\norm{F}_{L^2(\Omega)} + (1+\lambda)^{1/2} \norm{f}_{H^1(\partial\Omega)}\Big\},
\end{equation}
which is clearly far from sharp. However, as far as we know, this estimate is the only one we can derive for Lipschitz domains since the duality argument seems not applicable in this case. In other words, we cannot improve the convergence rate from $\omega_{k,\sigma}(\e)^{1/2}$ to $\omega_{k,\sigma}(\e)$ as we (and many other authors) have done in $C^{1,1}$ domains.
Actually, the optimal rate of convergence in Lipschitz domains is still an open problem even for periodic case. The best result so far in periodic case is contained in \cite{KLS}, where under additional symmetry condition on the coefficients the authors showed that the rate of convergence in $L^2$ is $O(\e|\ln \e|^{1/2+}) $. For almost-periodic homogenization, very little is known for convergence rate in Lipschitz domains. Nevertheless, estimate (\ref{ineq_L2_Lip}) still allows us to proceed with our work on uniform regularity.

To prove Theorem \ref{thm_H1_Lip}, the following energy estimate will be useful to us.
\begin{theorem}
	Assume $A$ satisfies the ellipticity condition (\ref{def_ellipticity}) and $\Omega$ is a bounded Lipschitz domain. Let $u$ be the weak solution of
	\begin{equation}
	\text{div}(A \nabla u) +\lambda u= F \quad \text{in } \Omega, \qquad \text{and} \qquad u =f \quad \text{on } \partial\Omega,
	\end{equation}
	with $\lambda \ge 0$, then	
	\begin{equation}\label{ineq_energy}
	\norm{\nabla u}_{L^2(\Omega)} + (1+\lambda)^{1/2}\norm{u}_{L^2(\Omega)} \le C (1+\lambda)^{-1/2}\norm{F}_{L^2(\Omega)} + C (1+\lambda)^{1/2}\norm{f}_{H^{1/2}(\partial \Omega)},
	\end{equation}
	where $C$ depends only on $d,m,\mu$ and $\Omega$.
\end{theorem}

\begin{proof}[Proof of Theorem \ref{thm_H1_Lip}]
	A direct algebraic manipulation shows that
	\begin{align}\label{eq_Lewe}
	\begin{aligned}
	\cL_\e(w_\e) + \lambda w_\e &= -\lambda \e \chi_{T,k}^\beta(x/\e) K_{\e,\delta} \left(\frac{\partial u_0^\beta}{\partial x_k}\right) + \frac{\partial}{\partial x_i} \left\{ b^{\alpha\beta}_{T,ij}(x/\e) K_{\e, \delta} \bigg( \frac{\partial u^\beta_0}{\partial x_j} \bigg) \right\}\\
	&\qquad  +  \frac{\partial}{\partial x_i} \left\{ \left\{\widehat{a}_{ij}^{\alpha\beta} - a^{\alpha\beta}_{ij}(x/\e) \right\} \left\{ K_{\e,\delta} \bigg(\frac{\partial u^\beta_0}{\partial x_j}\bigg)  - \frac{\partial u^\beta_0}{\partial x_j} \right\}\right\} \\
	&\qquad   + \e \frac{\partial}{\partial x_i} \left\{ a^{\alpha\beta}_{ij}(x/\e) \chi^{\beta\gamma}_{T,k}(x/\e) \frac{\partial}{\partial x_j}K_{\e, \delta}\bigg(  \frac{\partial u^\gamma_0}{\partial x_k }\bigg)  \right\},
	\end{aligned}
	\end{align}
	and
	\begin{align}\label{eq_bT}
	\begin{aligned}	
	&\frac{\partial}{\partial x_i} \left\{ b^{\alpha\beta}_{T,ij}(x/\e) K_{\e, \delta} \bigg( \frac{\partial u^\beta_0}{\partial x_j} \bigg) \right\} \\
	&= \ag{b_{T,ij}^{\alpha\beta}} \frac{\partial}{\partial x_i}K_{\e, \delta} \bigg( \frac{\partial u_0^\beta}{ \partial x_j} \bigg) + \frac{\partial}{\partial x_i} \left\{ T^{-2} \phi_{T,ij}^{\alpha\beta}(x/\e) K_{\e, \delta} \bigg(  \frac{\partial u_0^\beta}{\partial x_j}\bigg)\right\} \\
	&- \frac{\partial}{\partial x_i} \left\{ \frac{\partial}{\partial x_i} h_{T,j}^{\alpha\beta}(x/\e) K_{\e, \delta} \bigg( \frac{\partial u_0^\beta}{\partial x_j}\bigg)\right\} \\
	&- \e \frac{\partial}{\partial x_i} \left\{ \left[ \frac{\partial}{\partial x_k} (\phi_{T,ij}^{\alpha\beta})(x/\e) - \frac{\partial}{\partial x_i} (\phi_{T,kj}^{\alpha\beta})(x/\e) \right] \frac{\partial}{\partial x_k}K_{\e, \delta} \bigg( \frac{\partial u_0^\beta}{\partial x_j } \bigg) \right\},
	\end{aligned}
	\end{align}
	where $b_T$ and $\phi_T$ are defined by (\ref{def_bT}) and (\ref{def_phiT}), respectively. The proof of (\ref{eq_bT}) is based on the following observation derived from (\ref{def_phiT})
	\begin{equation}\label{eq_bT_phiT}
	b_{T,ij}^{\alpha\beta} = \ag{b_{T,ij}^{\alpha\beta}} - \frac{\partial}{\partial y_k} \left( \frac{\partial}{\partial y_k} \phi_{T,ij}^{\alpha\beta} - \frac{\partial}{\partial y_i} \phi_{T,kj}^{\alpha\beta} \right) - \frac{\partial}{\partial y_i}\left( \frac{\partial}{\partial y_k} \phi_{T,kj}^{\alpha\beta} \right) - T^{-2}\phi_{T,ij}^{\alpha\beta},
	\end{equation}
	as well as the fact that the second term in the right-hand side of (\ref{eq_bT_phiT}) is skew-symmetric with respect to $(i, k)$.
	
	Multiplying (\ref{eq_Lewe}) by $w_\e$ and integrating over $\Omega$, we arrive at
	\begin{equation}\label{ineq_dwdw}
	\begin{aligned}
	&\int_{\Omega} A(x/\e) \nabla w_\e \cdot \nabla w_\e + \lambda \int_{\Omega} w_\e \cdot w_\e \\
	& \qquad \le C\lambda \Theta_{k,\sigma}(\e) \norm{w_\e}_{L^2(\Omega)} \norm{K_{\e,\delta}(\nabla u_0)}_{L^2(\Omega)}\\
	& \qquad \qquad + C \norm{\nabla w_\e}_{L^2(\Omega)} \Big\{ \norm{K_{\e,\delta}(\nabla u_0) - \nabla u_0}_{L^2(\Omega)} + \Theta_{k,\sigma}(\e) \norm{\nabla K_{\e,\delta}(\nabla u_0) }_{L^2(\Omega)} \\
	& \qquad \qquad + |\ag{b_T}| \norm{K_{\e,\delta}(\nabla u_0)}_{L^2(\Omega)} + \Theta_{k,\sigma}(\e) \norm{K_{\e,\delta}(\nabla u_0)}_{L^2(\Omega)} \Big\},
	\end{aligned}
	\end{equation}
	where we also used (\ref{eq_bT}), Theorem \ref{thm_chiT_1}, \ref{thm_dualcor} and integration by parts. Note that there is no boundary integral arising since $w_\e \in H_0^1(\Omega)$.
	By the definition of $b_T$, we see that $|\ag{b_T}| \le C \norm{\nabla \chi_T - \psi}_{B^2}$ where $\psi$ is defined in (\ref{def_Vpot}).
	Thus it suffices to estimate $\norm{K_{\e,\delta}(\nabla u_0)}_{L^2(\Omega)}, \norm{K_{\e,\delta}(\nabla u_0) - \nabla u_0}_{L^2(\Omega)}$ and $\norm{\nabla K_{\e,\delta}(\nabla u_0) }_{L^2(\Omega)}$.
	
	First of all, by (\ref{ineq_S_p}) and energy estimate (\ref{ineq_energy}), one has
	\begin{align}\label{ineq_Kdu0}
	\begin{aligned}	
	\norm{K_{\e,\delta}(\nabla u_0)}_{L^2(\Omega)} &\le \norm{\nabla u_0}_{L^2(\Omega)} \\
	&\le  \frac{C}{(1+\lambda)^{1/2}}\norm{F}_{L^2(\Omega)} + C (1+\lambda)^{1/2}\norm{f}_{H^{1/2}(\partial \Omega)}.
	\end{aligned}
	\end{align}
	Next, observe that
	\begin{align*}
	&\norm{K_{\e,\delta}(\nabla u_0) - \nabla u_0}_{L^2(\Omega)} \\
	&\qquad \le \norm{S_{\e}(\theta_\delta \nabla u_0) - \theta_\delta \nabla u_0}_{L^2(\Omega)} + \norm{\nabla u_0}_{L^2(\Omega_{4\delta})} \\
	&\qquad  \le C \e \norm{\nabla^2 u_0}_{L^2(\Omega\setminus \Omega_\delta)} + C(1+\e \delta^{-1}) \norm{\nabla u_0}_{L^2(\Omega_{4\delta})},
	\end{align*}
	and
	\begin{equation*}
	\norm{\nabla K_{\e,\delta}(\nabla u_0) }_{L^2(\Omega)} \le C \norm{\nabla^2 u_0}_{L^2(\Omega\setminus \Omega_\delta)} + C\delta^{-1} \norm{\nabla u_0}_{L^2(\Omega_{4\delta})}.
	\end{equation*}
	Therefore, it is left to estimate $\norm{\nabla^2 u_0}_{L^2(\Omega\setminus \Omega_\delta)} $ and $\norm{\nabla u_0}_{L^2(\Omega_{4\delta})}$.
	
	To estimate $\norm{\nabla u_0}_{L^2(\Omega_{4\delta})}$, we write $u_0 = v + h$, where
	\begin{equation*}
	v(x) = \int_{\Omega} \Gamma_0(x-y)  (-\lambda u_0 + F(y))dy,
	\end{equation*}
	and $\Gamma_0$ denotes the matrix of fundamental solutions for homogenized operator $\cL_0$ in $\R^d$, with the pole at the origin. Note that $\cL_0 v= -\lambda u_0 + F$ and then by the well-known singular integral and fractional integral estimates,
	\begin{align*}
	\norm{ v}_{H^2(\Omega)} &\le  C\norm{F}_{L^2(\Omega)} + C\lambda \norm{u_0}_{L^2(\Omega)} \\
	& \le C(1+\lambda) \norm{f}_{H^1(\partial\Omega)} + C\norm{F}_{L^2(\Omega)}.
	\end{align*}
	Thus, by Lemma \ref{lem_Oeu_g},
	\begin{equation}\label{ineq_dv_delta}
	\norm{\nabla v}_{L^2(\Omega_{4\delta})} \le C\delta^{1/2} \norm{\nabla v}_{H^1(\Omega)} \le C\delta^{1/2} \Big\{\norm{F}_{L^2(\Omega)} + (1+\lambda) \norm{f}_{H^1(\partial\Omega)}\Big\}.
	\end{equation}
	Next we observe that $\cL_0 h = 0$ in $\Omega$, then
	\begin{align*}
	\norm{h}_{H^1(\partial\Omega)} &\le \norm{f}_{H^1(\partial\Omega)} + \norm{v}_{H^1(\partial\Omega)} \\
	&\le \norm{f}_{H^1(\partial\Omega)} + C\norm{v}_{H^2(\Omega)} \\
	& \le C(1+\lambda) \norm{f}_{H^1(\partial\Omega)} + C\norm{F}_{L^2(\Omega)}.
	\end{align*}
	Hence, it follows from the estimates for solutions of $L^2$ regularity problem in Lipschitz domains for the operator $\cL_0$ in \cite{DKV,Gwj} that
	\begin{equation*}
	\norm{(\nabla h)^*}_{L^2(\partial\Omega)} \le C(1+\lambda) \norm{f}_{H^1(\partial\Omega)} + C\norm{F}_{L^2(\Omega)},
	\end{equation*}
	where $(\nabla h)^*$ denotes the non-tangential maximal function of $\nabla h$. This, together with (\ref{ineq_dv_delta}), gives
	\begin{equation}\label{ineq_u0_delta}
	\norm{\nabla u_0}_{L^2(\Omega_{4\delta})} \le C\delta^{1/2} \Big\{\norm{F}_{L^2(\Omega)} + (1+\lambda) \norm{f}_{H^1(\partial\Omega)}\Big\}.
	\end{equation}
	
	It remains to estimate $\norm{\nabla^2 u_0}_{L^2(\Omega\setminus \Omega_\delta)}$. Note that the interior estimate for $\cL_0$ gives
	\begin{equation*}
	|\nabla^2 h(x)| \le \frac{C}{d(x)} \left( \fint_{B(x,\delta/4)} |\nabla h|^2 \right)^{1/2},
	\end{equation*}
	where $d(x) = \text{dist}(x,\partial\Omega)$. Hence,
	\begin{equation*}
	\norm{\nabla^2 h}_{L^2(\Omega\setminus \Omega_\delta)} \le C\delta^{-1/2} \Big\{\norm{F}_{L^2(\Omega)} + (1+\lambda) \norm{f}_{H^1(\partial\Omega)}\Big\}.
	\end{equation*}
	Combining this with the estimate for $\nabla^2 v$, one obtains
	\begin{equation*}
	\norm{\nabla^2 u_0}_{L^2(\Omega\setminus \Omega_\delta)} \le C\delta^{-1/2} \Big\{\norm{F}_{L^2(\Omega)} + (1+\lambda) \norm{f}_{H^1(\partial\Omega)}\Big\}.
	\end{equation*}
	As a result, we have shown
	\begin{equation}\label{ineq_Kdu0_du0}
	\norm{K_{\e,\delta}(\nabla u_0) - \nabla u_0}_{L^2(\Omega)} \le C(\e\delta^{-1/2} + \delta^{1/2}) \Big\{\norm{F}_{L^2(\Omega)} + (1+\lambda) \norm{f}_{H^1(\partial\Omega)}\Big\},
	\end{equation}
	and
	\begin{equation}\label{ineq_dKdu0}
	\norm{\nabla K_{\e,\delta}(\nabla u_0) }_{L^2(\Omega)} \le C\delta^{-1/2} \Big\{\norm{F}_{L^2(\Omega)} + (1+\lambda) \norm{f}_{H^1(\partial\Omega)}\Big\}.
	\end{equation}
	
	Now let $\delta = \Theta_{k,\sigma}(T)$. It follows from (\ref{ineq_dwdw}), (\ref{ineq_Kdu0}), (\ref{ineq_Kdu0_du0}) and (\ref{ineq_dKdu0}) that
	\begin{align*}
	&\int_{\Omega} A(x/\e) \nabla w_\e \cdot \nabla w_\e + \lambda \int_{\Omega} w_\e \cdot w_\e \\
	&\quad \le C\Theta_{k,\sigma}(T)^{1/2} \Big\{\norm{\nabla w_\e}_{L^2(\Omega)} + \lambda^{1/2} \norm{w_\e}_{L^2(\Omega)}\Big\}  \Big\{\norm{F}_{L^2(\Omega)} + (1+\lambda) \norm{f}_{H^1(\partial\Omega)}\Big\}\\
	&\quad \quad + C(1+\lambda)^{-1/2}\norm{\nabla \chi_T - \psi}_{B^2} \norm{\nabla w_\e}_{L^2(\Omega)} \Big\{\norm{F}_{L^2(\Omega)} + (1+\lambda) \norm{f}_{H^1(\partial\Omega)}\Big\}.
	\end{align*}
	It then follows that
	\begin{equation}
	\norm{\nabla w_\e}_{L^2(\Omega)} + \lambda^{1/2} \norm{w_\e}_{L^2(\Omega)} \le C\omega_{k,\sigma}(\e)^{1/2} \Big\{\norm{F}_{L^2(\Omega)} + (1+\lambda) \norm{f}_{H^1(\partial\Omega)}\Big\}.
	\end{equation}
	Finally, since $w_\e \in H_0^1(\Omega)$, it follows from the Poincar\'{e} inequality that
	\begin{equation}
	\norm{ w_\e}_{H^1(\Omega)} + (1+\lambda)^{1/2} \norm{w_\e}_{L^2(\Omega)} \le C\omega_{k,\sigma}(\e)^{1/2} \Big\{\norm{F}_{L^2(\Omega)} + (1+\lambda) \norm{f}_{H^1(\partial\Omega)}\Big\},
	\end{equation}
	which ends the proof.
\end{proof}

\begin{remark}
	A similar result for the Neumann problem also holds. Precisely, let $\Omega$ and $A$ be the same as in Theorem \ref{thm_H1_Lip} and $u_\e, u_0$ be the weak solution of (\ref{def_NP}) and the corresponding homogenized system with the same data, respectively, then
	\begin{equation}\label{ineq_L2N_Lip}
	\norm{u_\e - u_0}_{L^2(\Omega)} \le C\omega_{k,\sigma}(\e)^{1/2} \Big\{(1+\lambda)^{-1/2}\norm{F}_{L^2(\Omega)} + (1+\lambda)^{1/2} \norm{g}_{L^2(\partial\Omega)}\Big\}.
	\end{equation}
\end{remark}

Although most of this section was focused on Lipschitz domains with $H^1$ Dirichlet boundary data, sometimes we are also interested in $H^s$ boundary data when $s\neq 1$. The next theorem is concerned with the rate of convergence in $C^{1,1}$ domains with $H^s(\partial\Omega)$ Dirichlet boundary data, where $1/2\le s\le 3/2$. This theorem is of independent interest, though it will not be used in this paper. Before stating the theorem, we recall that Theorem \ref{thm_L2_C11} actually shows that if $\Omega$ is a $C^{1,1}$ domain, then
\begin{equation}\label{ineq_L2_C11}
\norm{u_\e - u_0}_{L^2(\Omega)} \le C\omega_{k,\sigma}(\e) \bigg\{\norm{F}_{L^2(\Omega)} + (1+\lambda) \norm{f}_{H^{3/2}(\partial\Omega)}\bigg\}.
\end{equation}
This follows from (\ref{ineq_L2DN_C11}) and the energy estimate.

\begin{theorem}\label{thm_L2_C11H1}
	Suppose that $\Omega$ is a bounded $C^{1,1}$ domain and $A\in APW^2(\R^d)$ satisfies the ellipticity condition (\ref{def_ellipticity}). Let $u_\e$ be the weak solution  of (\ref{def_NP}) and $u_0$ be the weak solution of homogenized system with the same data $(F,f)$. Then for every $s\in [1/2,3/2]$,
	\begin{equation}
	\Norm{u_\e - u_0}_{L^2(\Omega)} \le  C\omega_{k,\sigma}(\e) \norm{F}_{L^2(\Omega)} + C\big\{\omega_{k,\sigma}(\e)(1+\lambda)\big\}^{s-1/2}\norm{f}_{H^s(\partial\Omega)},
	\end{equation}
	where $C$ depends only on $d,m,\sigma, A$ and $\Omega$.
\end{theorem}

\begin{proof}
	Since $f\in H^s(\partial\Omega), s\ge 1/2$ and $\Omega$ is $C^{1,1}$, then there exists a extension operator $E$ such that $Ef\in H^{s+1/2}(\R^d)$ and $\text{Tr} (Ef) = f$ on $\partial\Omega$ and $\norm{Ef}_{H^{s+1/2}(\R^d)} \le C\norm{f}_{H^s(\partial\Omega)}$, where $C$ depends only on $d$ and $\Omega$. Denote $Ef$ by $\tilde{f}$. Let $\phi\in C^\infty_0(B_1(0))$ such that $\int \phi = 1$ and $\phi_\delta(x) = \delta^{-d} \phi(x/\delta)$, where $\delta>0$ is to be determined. Set $\tilde{f}_\delta = \phi_\delta * \tilde{f}$. Clearly, $\tilde{f}_\delta$ is smooth. We claim that
	\begin{equation}\label{ineq_Hs_Fourier}
	\norm{\tilde{f}_\delta}_{H^2(\R^d)} \le C\delta^{s-3/2} \norm{\tilde{f}}_{H^{s+1/2}(\R^d)} \le C\delta^{s-3/2} \norm{f}_{H^s(\partial\Omega)}.
	\end{equation}
	Actually, this is a standard exercise for the equivalent $H^s$ norm defined by Fourier transform, i.e.
	\begin{equation*}
	\norm{g}_{H^s(\R^d)}^2 = \int_{\R^d} (1+|\xi|^2)^{s} |\mathcal{F} g(\xi)|^2 d\xi.
	\end{equation*}
	The details are left to the readers. Now we let $f_\delta = \text{Tr} \tilde{f}_\delta $. By trace theorem and (\ref{ineq_Hs_Fourier}), we know $\norm{f_\delta}_{H^{3/2}(\partial\Omega)} \le C\delta^{s-3/2} \norm{f}_{H^s(\partial\Omega)}$.
	
	Next, we construct a Dirichlet problem as follows:
	\begin{equation*}
	\cL_\e(\tilde{u}_\e) + \lambda \tilde{u}_\e = F \quad \text{in } \Omega, \qquad \text{and} \qquad \tilde{u}_\e =f_\delta \quad \text{on } \partial\Omega.
	\end{equation*}
	Also, the corresponding homogenized problem is:
	\begin{equation*}
	\cL_0(\tilde{u}_0)  + \lambda \tilde{u}_0 = F \quad \text{in } \Omega, \qquad \text{and} \qquad \tilde{u}_0 =f_\delta \quad \text{on } \partial\Omega.
	\end{equation*}
	Since $\Omega$ is $C^{1,1}$ and $f_\delta \in H^{3/2}(\partial\Omega)$, then it follows form (\ref{ineq_L2_C11}) that
	\begin{align*}
	\norm{\tilde{u}_\e - \tilde{u}_0}_{L^2(\Omega)} & \le C\omega_{k,\sigma}(\e) \big\{ \norm{F}_{L^2(\Omega)} + (1+\lambda)\norm{f_\delta}_{H^{3/2}(\partial\Omega)} \big\} \\
	& \le C\omega_{k,\sigma}(\e) \big\{ \norm{F}_{L^2(\Omega)} + \delta^{s-3/2}(1+\lambda) \norm{f}_{H^{s}(\partial\Omega)} \big\}.
	\end{align*}
	
	On the other hand, $v_\e = u_\e - \tilde{u}_\e$ satisfies
	\begin{equation*}
	\cL_\e(v_\e)  + \lambda v_\e = 0 \quad \text{in } \Omega, \qquad \text{and} \qquad v_\e =f - f_\delta \quad \text{on } \partial\Omega.
	\end{equation*}
	Then it follows from energy estimate and trace theorem that
	\begin{equation*}
	\norm{u_\e - \tilde{u}_\e}_{L^2(\Omega)} \le C \norm{f -f_\delta}_{H^{1/2}(\partial\Omega)} \le C\norm{\tilde{f} - \tilde{f}_\delta}_{H^1(\R^d)} \le C \delta^{s-1/2} \norm{f}_{H^s(\partial\Omega)},
	\end{equation*}
	where we have used the fact $s\ge 1/2$. Similarly, we also have
	\begin{equation*}
	\norm{u_0 - \tilde{u}_0}_{L^2(\Omega)} \le C\delta^{s-1/2} \norm{f}_{H^s(\partial\Omega)}.
	\end{equation*} 
	As a consequence,
	\begin{align*}
	&\norm{u_\e - u_0}_{L^2(\Omega)} \\
	&\qquad  \le \norm{u_\e - \tilde{u}_\e}_{L^2(\Omega)} + \norm{\tilde{u}_\e - \tilde{u}_0}_{L^2(\Omega)} + \norm{u_0 - \tilde{u}_0}_{L^2(\Omega)} \\
	&\qquad \le C \omega_{k,\sigma}(\e) \norm{F}_{L^2(\Omega)} + C\big\{ \delta^{s-1/2}  + \delta^{s-3/2} \omega_{k,\sigma}(\e) (1+\lambda) \big\} \norm{f}_{H^s(\partial\Omega)} \\
	&\qquad \le C \omega_{k,\sigma}(\e) \norm{F}_{L^2(\Omega)} + C \big\{\omega_{k,\sigma}(\e)(1+\lambda)\big\}^{s-1/2} \norm{f}_{H^s(\partial\Omega)},
	\end{align*}
	where in the last inequality we have chosen $\delta = \omega_{k,\sigma}(\e)(1+\lambda)$.
\end{proof}

\section{Boundary Lipschitz estimate}
In this section we will study the uniform boundary Lipschitz estimates down to the scale $\e$ in $C^{1,\alpha}$ domains with Dirichlet or Neumann boundary conditions. The Dirichlet and Neumann cases will be treated in two subsections separately. We modify the argument in \cite{Shen1} to make it adapted to general $\lambda > 0$.

Let $D_r, \Delta_r$ be defined in (\ref{def_C1a}). Note that $D_r$ acts as a subset of $\Omega$ who shares the same boundary portion $\Delta_r $ with $D_r$. Therefore, to establish the boundary estimates, it suffices to consider the boundary value problems in $D_r$. Throughout this section, $\alpha \in (0,1)$ will be fixed and $\lambda $ is restricted in $[0,1]$ so that it essentially has no influence on our proofs and results. For the case $\lambda >1$, we can use rescaling $v_\e(x) = \lambda u_\e (\lambda^{-1/2}x)$ so that it reduces to the case of $\lambda = 1$. However, in this case the constant will also depend on $\lambda$.

\subsection{Dirichlet boundary value problems}
Throughout this subsection, we let $u_\e \in H^1(D_2;\R^d)$ be a weak solution of $\cL_\e(u_\e) + \lambda u_\e = F$ in $D_2$ with $u_\e = f$ on $\Delta_2$. Define the following auxiliary quantities adapted for nonzero $\lambda$:
\begin{equation}\label{def_Phi}
\begin{aligned}
\Phi(t) &= \frac{1}{t} \inf_{q\in \R^d} \Bigg\{ \left( \fint_{D_t} |u_\e - q|^2 \right)^{1/2} + t^2 \left( \fint_{D_t} |F|^p \right)^{1/p} \\
&\qquad + t^2\lambda |q| + \norm{f -q}_{L^\infty(\Delta_{t})} + t \norm{\nabla_{\tan} f}_{L^\infty(\Delta_t)}\Bigg\},
\end{aligned}
\end{equation}
and
\begin{equation}\label{def_H}
\begin{aligned}
H(t;u) & = \frac{1}{t} \inf_{\substack{P\in \R^{d\times d} \\q\in \R^d}} \Bigg\{ \left( \fint_{D_t} |u -Px - q|^2 \right)^{1/2} + t^2 \left( \fint_{D_t} |F|^p \right)^{1/p} \\
&\qquad + t^2 \lambda \norm{Px+q}_{L^\infty(D_t)} + \norm{f -Px-q}_{L^\infty(\Delta_{t})}  \\
&\qquad + t \norm{\nabla_{\tan} (f - Px) }_{L^\infty(\Delta_t)} + t^{1+\tau} \norm{\nabla_{\tan} (f - Px) }_{C^\tau(\Delta_t)} \Bigg\},
\end{aligned}
\end{equation}
where $p>d$ and $\tau \in (0,\alpha)$.

\begin{lemma}\label{lem_ue_v}
	Let $\e \le r\le 1$. There exists $v\in H^1(D_r;\R^d)$ such that $\cL_0 (v)  + \lambda v = F$ in $D_r$, $v =f$ on $\Delta_r$ and
	\begin{equation}\label{ineq_ue_flat}
	\frac{1}{r} \left( \fint_{D_r} |u_\e - v|^2\right)^{1/2} \le C[\omega_{k,\sigma}(\e/r)]^{1/2} \Phi(2r),
	\end{equation}
	where $C$ depends only on $A,\sigma$ and $M$.
\end{lemma}
\begin{proof}
	By rescaling, it is sufficient to prove (\ref{ineq_ue_flat}) with $r = 1$. First by Caccioppoli's inequality,
	\begin{equation*}
	\int_{D_{3/2}} |\nabla u_\e|^2 \le C \left\{ \int_{D_{2}} |u_\e|^2 + \int_{D_{2}} |F|^2 + \norm{f}^2_{L^\infty(\Delta_2)} + \norm{\nabla_{\tan} f}^2_{L^\infty(\Delta_2)}\right\}.
	\end{equation*}
	By the co-area formula, this implies that there exists some $t\in [5/4,3/2]$ such that
	\begin{align*}
	&\int_{\partial D_t\setminus \Delta_2} \left( |\nabla u_\e|^2 + |u_\e|^2 \right) \\
	& \qquad \le C \left\{ \int_{D_{2}} |u_\e|^2 + \int_{D_{2}} |F|^2 + \norm{f}^2_{L^\infty(\Delta_2)} + \norm{\nabla_{\tan} f}^2_{L^\infty(\Delta_2)}\right\}.
	\end{align*}
	Let $v$ be the weak solution to the Dirichlet problem:
	\begin{equation}
	\cL_0(v) + \lambda v = F \quad \text{in } D_t \quad \text{and} \quad v = u_\e \quad \text{on } \partial D_t.
	\end{equation}
	Note that $v = f$ on $\Delta_1$. Then it follows from (\ref{ineq_L2_Lip}) that
	\begin{align*}
	&\norm{u_\e - v}_{L^2(D_1)} \le \norm{u_\e - v}_{L^2(D_t)} \\
	&\qquad  \le C [\omega_{k,\sigma}(\e)]^{1/2} \left\{ \norm{F}_{L^2(D_t)} + \norm{u_\e}_{H^1(\partial D_t)} \right\}\\
	&\qquad  \le C [\omega_{k,\sigma}(\e)]^{1/2} \left\{ \norm{u_\e}_{L^2(D_2)} +\norm{F}_{L^2(D_2)} + \norm{f}_{L^\infty(\Delta_2)} +  \norm{\nabla_{\tan} f}^2_{L^\infty(\Delta_2)} \right\}.
	\end{align*}
	This implies
	\begin{align*}
	\left( \fint_{D_1} |u_\e - v|^2\right)^{1/2} &\le C[\omega_{k,\sigma}(\e)]^{1/2} \Bigg\{ \left( \fint_{D_2} |u_\e|^2 \right)^{1/2} + \left( \fint_{D_2} |F|^p \right)^{1/p} \\
	&\qquad + \norm{f}_{L^\infty(\Delta_{2})} +  \norm{\nabla_{\tan} f}_{L^\infty(\Delta_2)}\Bigg\}.
	\end{align*}
	Finally observe that the last inequality still holds if we subtract a constant $q\in \R^d$ simultaneously from $u_\e, v$ and $f$. This gives us the desired estimate with $r =1$ by taking the infimum over all $q\in \R^d$.
\end{proof}

\begin{lemma}[Flatness property for $\cL_0$]\label{lem_Htheta}
	Let $v \in H^1(D_2;\R^d)$ be a weak solution of $\cL_0(v) + \lambda v= F$ in $D_2$ with $ v = f$ on $\Delta_2$. Then there exists $\theta\in (0,1/4)$, depending only on $p,A,\tau,\alpha$ and $M$, such that
	\begin{equation}
	H(\theta r; v) \le \frac{1}{2} H(r; v).
	\end{equation}
\end{lemma}
\begin{proof}
	This lemma is similar as \cite[Theorem 7.1]{AS}, which follows form the boundary $C^{1,\alpha}$ estimate for the second-order elliptic system with constant coefficients. By rescaling, we may assume $r = 1$. By choosing $q = v(0)$ and $P = \nabla v(0)$, we can see
	\begin{equation*}
	H(\theta;v) \le C \theta^\beta \norm{v}_{C^{1,\beta}(D_\theta)} + C\theta^\beta \left(\fint_{D_1} |F|^p \right)^{1/p},
	\end{equation*}
	where $\beta = \min\{\tau, (p-d)/p\}$. By boundary $C^{1,\alpha}$ estimate, we obtain
	\begin{align*}
	\norm{v}_{C^{1,\beta}(D_\theta)} \le & C \bigg \{ \left( \fint_{D_1} |v|^2 \right)^{1/2} + \left(\fint_{D_1} |F|^p \right)^{1/p}  \\
	& \qquad + \norm{f}_{L^\infty(\Delta_1)} + \norm{\nabla_{\tan} f}_{L^\infty(\Delta_1) } + \norm{\nabla_{\tan} f}_{C^\beta (\Delta_1) } \bigg\}.
	\end{align*}
	Then it follows that
	\begin{align}
	\begin{aligned}\label{ineq_H_thev}
	H(\theta;v) \le & C \theta^\beta \bigg \{ \left( \fint_{D_1} |v|^2 \right)^{1/2} + \left(\fint_{D_1} |F|^p \right)^{1/p}  \\
	& \qquad + \norm{f}_{L^\infty(\Delta_1)} + \norm{\nabla_{\tan} f}_{L^\infty(\Delta_1) } + \norm{\nabla_{\tan} f}_{C^\beta (\Delta_1) } \bigg\}.
	\end{aligned}
	\end{align}
	Now note that $\cL_0(Px+q) = 0$ for any $P\in \R^{d\times d}, q\in \R^d$. Let $w = v -Px-q$, then
	\begin{equation}
	\cL(w)+\lambda w = F - \lambda(Px+q),
	\end{equation}
	with Dirichlet boundary data $w = f-Px-q$ on $\Delta_2$. Applying (\ref{ineq_H_thev}) to $w$, we arrive at
	\begin{align}
	\begin{aligned}\label{ineq_H_thew}
	H(\theta;w) \le & C \theta^\beta \bigg \{ \left( \fint_{D_1} |v-Px-q|^2 \right)^{1/2} + \left(\fint_{D_1} |F|^p \right)^{1/p}   \\
	& \qquad + \lambda \norm{Px+q}_{L^\infty(D_1)} + \norm{f-Px-q}_{L^\infty(\Delta_1)} \\
	& \qquad + \norm{\nabla_{\tan} (f-Px)}_{L^\infty(\Delta_1) } + \norm{\nabla_{\tan} (f-Px)}_{C^\beta (\Delta_1) } \bigg\}.
	\end{aligned}
	\end{align}
	Also, it follows from triangle inequality that
	\begin{equation}\label{ineq_HvHw}
	H(\theta;v) \le H(\theta;w) + \theta \lambda \norm{Px+q}_{L^\infty(D_1)}.
	\end{equation}
	Combing (\ref{ineq_H_thew}) and (\ref{ineq_HvHw}) and taking the infimum over all $P\in \R^{d\times d}, q\in \R^d$, we obtain
	\begin{equation}
	H(\theta;v) \le C\theta^\beta H(1;v).
	\end{equation}
	The desired estimate follows by fixing a $\theta \in (0,1/4)$ so small that $C\theta^\beta \le 1/2$.
\end{proof}

\begin{lemma}[Flatness property for $\cL_\e$]\label{lem_HPhi}
	Let $0<\e<1/2$, then there exists a $\theta\in (0,1/4)$ such that for any $r\in [\e,1/2]$,
	\begin{equation}
	H(\theta r; u_\e) \le \frac{1}{2} H(r;u_\e) + C [\omega_{k,\sigma}(\e/r)]^{1/2} \Phi(2r),
	\end{equation}
	where $C$ depends only on $p,A,\alpha,\tau,\sigma$ and $M$.
\end{lemma}
\begin{proof}
	Fix $r\in [\e,1/2]$. Let $v$ be a solution of $\cL_0(v) + \lambda v= F$ in $D_r$ with $v = f$ on $\Delta_r$. Observe that
	\begin{align*}
	H(\theta r; u_\e) &\le  H(\theta r; v) + \frac{1}{\theta r}\left( \fint_{D_{\theta r}} |u_\e - v|^2 \right)^{1/2} \\
	&\le  \frac{1}{2} H(r; v) + \frac{1}{\theta r} \left( \fint_{D_{\theta r}} |u_\e - v|^2 \right)^{1/2}\\
	& \le  \frac{1}{2} H(r; v) + \frac{C}{ r} \left( \fint_{D_{r}} |u_\e - v|^2 \right)^{1/2} \\
	&\le  \frac{1}{2} H(r; v) + C[\omega_{k,\sigma}(\e/r)]^{1/2} \Phi(2r)
	\end{align*}
	where we have used Lemma \ref{lem_Htheta} for the second inequality and Lemma \ref{lem_ue_v} for the last inequality.
\end{proof}

\begin{lemma}\label{lem_iter}
	Let $H(t)$ and $h(t)$ be two nonnegative continuous functions on the interval $(0,1]$. Suppose that there exists a constant $C_0$ such that
	\begin{equation}\label{ineq_iter_c1}
	\max_{r\le t\le 2r} H(t) + \max_{r\le t,s \le 2r} |h(t) - h(s)| \le C_0 H(2r).
	\end{equation}
	for any $r\in [\e ,1/2]$. We further assume that
	\begin{equation}\label{ineq_iter_c2}
	H(\theta r) \le \frac{1}{2} H(r) + C_0 \eta(\e/r) \left\{H(2r) + h(2r)\right\},
	\end{equation}
	for any $r\in [\e,1/2]$, where $\theta\in (0,1/4)$ and $\eta$ is a nonnegative increasing function on $[0,1]$ satisfying $\eta(0) = 0$ and
	\begin{equation}\label{ineq_iter_c3}
	\int_0^1 \frac{\eta(t)}{t} dt < \infty.
	\end{equation}
	Then
	\begin{equation}
	\max_{\e \le r\le 1} \left\{ H(r) + h(r) \right\} \le C \left\{ H(1) + h(1) \right\},
	\end{equation}
	where $C$ depends only on $C_0,\theta$ and $\eta$.
\end{lemma}

This lemma was proved in \cite[Lemma 8.5]{Shen1}, where the Dini-type condition (\ref{ineq_Dini_Intro}) is involved.

\begin{proof}[Proof of Theorem \ref{thm_Lip_DP}]
	We assume that $0<\e < 1/4$ and let $u_\e$ define on $D_2$ as before. For $r\in (0,1)$ and $t\in (r,2r)$, it is easy to see that $H(t;u_\e) \le C H(2r;u_\e)$.
	
	Next, we let $h(r) = |P_r|$, where $P_r$ is the $d\times d$ matrix such that
	\begin{equation}
	\begin{aligned}
	H(r;u_\e) & = \frac{1}{r} \inf_{q\in \R^d} \Bigg\{ \left( \fint_{D_r} |u_\e -P_r x - q|^2 \right)^{1/2} + r^2 \left( \fint_{D_r} |F|^p \right)^{1/p} \\
	&\qquad + r^2 \lambda \norm{P_r x+ q}_{L^\infty(D_{r})} + \norm{f -P_r x-q}_{L^\infty(\Delta_{r})} \\
	&\qquad + r \norm{\nabla_{\tan} (f - P_r x) }_{L^\infty(\Delta_r)}  + r^{1+\tau} \norm{\nabla_{\tan} (f - P_r x) }_{C^\tau(\Delta_r)} \Bigg\},
	\end{aligned}
	\end{equation}
	Let $t,s\in [r,2r]$. Using
	\begin{align*}
	|P_t - P_s| &\le \frac{C}{r} \inf_{q\in \R^d} \left( \fint_{D_r} |(P_t-P_s)x - q|^2 \right)^{1/2} \\
	& \le \frac{C}{t} \inf_{q\in \R^d} \left( \fint_{D_t} |u_\e - P_tx - q|^2 \right)^{1/2} + \frac{C}{s} \inf_{q\in \R^d} \left( \fint_{D_s} |u_\e - P_sx - q|^2 \right)^{1/2} \\
	& \le CH(t;u_\e) + CH(s;u_\e);\\
	& \le CH(2r;u_\e).
	\end{align*}
	Thus, we obtain
	\begin{equation}
	\max_{r\le t,s \le 2r} |h(t) - h(s)| \le C H(2r;u_\e).
	\end{equation}
	Furthermore, by the definition of $\Phi$ and $H$,
	\begin{equation}
	\Phi(2r) \le H(2r;u_\e) + h(2r).
	\end{equation}
	In view of Lemma \ref{lem_HPhi}, we have
	\begin{equation}
	H(\theta r; u_\e) \le \frac{1}{2} H(r;u_\e) + C [\omega_{k,\sigma}(\e/r)]^{1/2} \{ H(2r;u_\e) + h(2r) \},
	\end{equation}
	for all $r\in [\e ,1/2]$. Note that the function $H(r) = H(r;u_\e)$ and $h(r)$ satisfies the conditions (\ref{ineq_iter_c1}), (\ref{ineq_iter_c2}) and (\ref{ineq_iter_c3}). Then by Lemma \ref{lem_iter}, we obtain that for all $r \in [\e,1/2]$,
	\begin{align*}
	\inf_{q\in \R^d} \frac{1}{r} \left( \fint_{D_r} |u_\e - q|^2 \right)^{1/2} &\le C\{ H(r;u_\e) + h(r) \} \\
	& \le C\{ H(1;u_\e) + h(1)\}\\
	& \le C\left\{ \left( \fint_{D_1} |u_\e|^2 \right)^{1/2} + \norm{f}_{C^{1,\tau}(\Delta_1)} + \norm{F}_{L^p(D_1)} \right\}\\
	& \le C\left\{ \left( \fint_{D_1} |\nabla u_\e|^2 \right)^{1/2} + \norm{f}_{C^{1,\tau}(\Delta_1)} + \norm{F}_{L^p(D_1)} \right\},
	\end{align*}
	where we have used the Poincar\'{e} inequality and the fact $u_\e = f$ on $\Delta_1$ in the last inequality. This, together with the Caccioppoli inequality, gives the estimate (\ref{ineq_DP_Lip}).
\end{proof}

\begin{remark}
	It is obvious to see that the argument above for the large scale boundary Lipschitz estimate also works for the interior Lipschitz estimate; see \cite[Theorem 11.1]{ShenZhuge2} for another proof. Indeed, we are able to establish
	\begin{equation}\label{ineq_int_Lip}
	\left( \fint_{B_r} |\nabla u_\e|^2 \right)^{1/2} \le C\left\{ \left( \fint_{B_1} |\nabla u_\e|^2 \right)^{1/2} + \norm{F}_{L^p(B_1)} \right\},
	\end{equation}
	where $u_\e$ is a solution for $\cL_\e u_\e + \lambda u_\e = F$ in $B_2$.
\end{remark}

\begin{remark}\label{rmk_Lip}
	As we have mentioned in the Introduction, under the additional condition of smoothness on the coefficients, the full uniform boundary Lipschitz estimate follows from Theorem \ref{thm_Lip_DP} and a blow-up argument. In fact, it is sufficient to assume $A$ is H\"{o}lder continuous, i.e., there exist $\delta>0$ and $C$ such that
	\begin{equation}\label{ineq_A_holder}
	|A(x) - A(y)| \le C|x-y|^\delta,
	\end{equation}
	for all $x,y\in \R^d$. 
	
	Now we would like to give the details of the blow-up argument. Let $u_\e$ be as before. Set $\tilde{u}(x) = \e^{-1} u_\e(x/\e)$, then $\tilde{u}$ satisfies
	\begin{equation}
	- \text{div} (A(x)\nabla \tilde{u}(x)) = F_\e(x) \quad \text{in } \tilde{D}_1, \qquad \text{and} \qquad \tilde{u} = f_\e \quad \text{on } \tilde{\Delta}_1,
	\end{equation}
	where $F_\e(x) = \e F(\e x)$ and $f_\e(x) = \e^{-1}f(\e x)$ and
	\begin{equation}
	\begin{aligned}
	\tilde{D}_1 &= \left\{  (x',x_d)\in \R^d: |x'|<1 \text{ and } \e^{-1} \phi(\e x') < x_d < \e^{-1}\phi(\e x') + 1 \right\}, \\
	\tilde{\Delta}_1 &= \left\{  (x',x_d)\in \R^d: |x'|<1 \text{ and } x_d = \e^{-1} \phi(\e x') \right\}. \\
	\end{aligned}
	\end{equation}
	Recall that $\phi(0) = 0$ and $\norm{\nabla \phi}_{C^\tau(\R^{d-1})} \le M$. Let $\phi_\e(x) = \e^{-1}\phi(\e x)$. Then we also have $\phi_\e(0) = 0$ and $\norm{\nabla \phi_\e}_{C^\tau(B(0,1))} \le M$. Without loss of generality, we can also assume $f(0) = 0$ by subtracting a constant from the solution since we are only concerned with the magnitude of the gradient. So we have $\norm{f_\e}_{C^{1,\alpha}(\tilde{\Delta}_1)} \le \norm{f}_{C^{1,\alpha}(\Delta_1)}$. Moreover, it is clear that $\norm{F_\e}_{L^p(\tilde{D}_1)} \le \norm{F}_{L^p(D_1)}$ for $p> d$. If $A$ satisfies (\ref{ineq_A_holder}), then we can apply the Lipschitz estimate (or $C^{1,\alpha}$ estimate) for $\tilde{u}$ and obtain
	\begin{align*}
	|\nabla \tilde{u}(0)| &\le C\left\{ \left( \fint_{\tilde{D}_1} |\nabla \tilde{u}|^2 \right)^{1/2} + \norm{f_\e}_{C^{1,\tau}(\tilde{\Delta}_1)} + \norm{F_\e}_{L^p(\tilde{D}_1)} \right\} \\
	& \le C\left\{ \left( \fint_{D_\e} |\nabla u_\e|^2 \right)^{1/2} + \norm{f}_{C^{1,\tau}(\Delta_1)} + \norm{F}_{L^p(D_1)} \right\}.
	\end{align*}
	
	Now noting that $\nabla u_\e(0) = \nabla \tilde{u}(0)$ and combining the last inequality with Theorem \ref{thm_Lip_DP}, we obtain
	\begin{equation}
	|\nabla u_\e(0)| \le C \left\{ \left( \fint_{D_1} |\nabla u_\e|^2 \right)^{1/2} + \norm{f}_{C^{1,\tau}(\Delta_1)} + \norm{F}_{L^p(D_1)} \right\}.
	\end{equation}
	Observe that this argument works equally well for the points whose distance from boundary is less than $\e$. On the other hand, for those points far away from boundary, we can combine the large scale interior Lipschitz estimate (\ref{ineq_int_Lip}) and the blow-up argument to obtain the full uniform Lipschitz estimate. As a consequence, we obtain the following.
\end{remark}

\begin{theorem}[Global Lipschitz estimate for DP]
	Let $\Omega$ be a bounded $C^{1,\alpha}$ domain. Suppose that $A\in APW^2(\R^d)$ satisfies the ellipticity condition (\ref{def_ellipticity}) and H\"{o}lder continuity (\ref{ineq_A_holder}). Moreover, $\omega_{k,\sigma}$ obeys the Dini-type condition (\ref{ineq_Dini_Intro}) for some $\sigma \in (0,1), k\ge 1$. If $u_\e$ is the weak solution of (\ref{def_DP}) with $F\in L^p(\Omega), p>d$ and $f\in C^{1,\tau}(\partial\Omega), \tau > 0$, then
	\begin{equation*}
	\norm{\nabla u_\e}_{L^\infty(\Omega)} \le C \left\{ \norm{f}_{C^{1,\tau}(\partial\Omega)} + \norm{F}_{L^p(\Omega)} \right\},
	\end{equation*}
	where the constant is independent of $\e$.
\end{theorem}

Finally, we mention in advance that we should be able to obtain the full Lipschitz estimate for Neumann problems, as well as full H\"{o}lder estimates (Section 5.3) for both Dirichlet and Neumann problems, by the same blow-up argument. The details are left to the readers. (For H\"{o}lder estimates, it is sufficient to assume that $A$ belongs to VMO space \cite{Shen1}.)

\subsection{Neumann boundary value problems}
Actually, Neumann problems are treated analogously as Dirichlet problems. All the lemmas and results are parallel to those proved for Dirichlet problems. For this reason, we will just list all the lemmas needed as a sketch of the proof and omit all the technical details. Throughout this subsection, we let $u_\e \in H^1(D_2;\R^d)$ be a weak solution of $\cL_\e(u_\e) + \lambda u_\e= F$ in $D_2$ with $\partial u_\e/\partial \nu_\e = g$ on $\Delta_2$. Define the following auxiliary quantities:
\begin{equation}
\begin{aligned}
\Psi(t) = & \frac{1}{t} \inf_{q\in \R^d} \Bigg\{ \left( \fint_{D_t} |u_\e - q|^2 \right)^{1/2} + t^2 \left( \fint_{D_t} |F|^p \right)^{1/p} \\
&\qquad +  t^2\lambda |q| + t \norm{g}_{L^\infty(\Delta_t)}\Bigg\},
\end{aligned}
\end{equation}
and
\begin{equation}
\begin{aligned}
J(t;u) & = \frac{1}{t} \inf_{\substack{P\in \R^{d\times d} \\q\in \R^d}} \Bigg\{ \left( \fint_{D_t} |u -Px - q|^2 \right)^{1/2} + t^2 \left( \fint_{D_t} |F|^p \right)^{1/p}\\
&\qquad  + t^2 \lambda \norm{Px+q}_{L^\infty(D_t)} + t \Norm{g - \frac{\partial}{\partial \nu_0}(Px) }_{L^\infty(\Delta_t)} \\
&\qquad + t^{1+\tau} \Norm{g - \frac{\partial}{\partial \nu_0}(Px) }_{C^\tau(\Delta_t)} \Bigg\},
\end{aligned}
\end{equation}
where $p>d$ and $\tau\in (0,\alpha)$.

\begin{lemma}\label{lem_Neu_rate}
	Let $\e \le r\le 1$. There exists $w\in H^1(D_r;\R^d)$ such that $\cL_0 (w) + \lambda w = F$ in $D_r$, $\partial w/\partial \nu_0 =g$ on $\Delta_r$ and
	\begin{equation}\label{ineq_ue_flatN}
	\left( \fint_{D_r} |u_\e - w|^2\right)^{1/2} \le C[\omega_{k,\sigma}(\e/r)]^{1/2} \Psi(2r),
	\end{equation}
	where $C$ depends only on $A,\sigma$ and $M$.
\end{lemma}

\begin{proof}
	The lemma follows from (\ref{ineq_L2N_Lip}) and the same argument as Lemma \ref{lem_ue_v}.
\end{proof}

\begin{lemma}\label{lem_Neu_L0}
	Let $w \in H^1(D_2;\R^d)$ be a weak solution of $\cL_0(w) + \lambda w = F$ in $D_2$ with $\partial w/\partial \nu_0 = g$ on $\Delta_2$. Then there exists $\theta\in (0,1/4)$, depending only on $p,A,\tau,\alpha$ and $M$, such that
	\begin{equation}
	J(\theta r; w) \le \frac{1}{2} J(r; w).
	\end{equation}
\end{lemma}
\begin{proof}
	This lemma follows from the boundary $C^{1,\alpha}$ estimate for Neumann problems and the similar argument as Lemma \ref{lem_Htheta}.
\end{proof}

\begin{lemma}\label{lem_Neu_Le}
	Let $0<\e<1/2$, then there exists a $\theta\in (0,1/4)$ such that for any $r\in [\e,1/2]$,
	\begin{equation}
	J(\theta r; u_\e) \le \frac{1}{2} J(r;u_\e) + C [\omega_{k,\sigma}(\e/r)]^{1/2} \Psi(2r),
	\end{equation}
	where $C$ depends only on $p,A,\alpha,\tau,\delta$ and $M$.
\end{lemma}

\begin{proof}
	The lemma follows from the same lines as Lemma \ref{lem_HPhi}, by combining Lemma \ref{lem_Neu_rate} and \ref{lem_Neu_L0}.
\end{proof}

\begin{proof}[Proof of Theorem \ref{thm_Lip_NP}]
	With Lemma \ref{lem_Neu_Le} at our disposal, (\ref{ineq_NP_Lip}) follows from Lemma \ref{lem_iter}, as in the case of Dirichlet boundary conditions. We omit the details.
\end{proof}

\section{Applications}
\subsection{Improved estimates of approximate correctors} 
In \cite{ShenZhuge2}, we obtained the interior uniform H\"{o}lder continuity of the system
\begin{equation}
-\text{div}(A(x/\e) \nabla u_\e) + \lambda u_\e = F + \text{div} f,
\end{equation}
down to the scale $\e$, by a compactness argument. Based on that we were able to establish the estimates (\ref{ineq_dchiT_1}) for the approximate correctors for any $\sigma>0$. However, to recover the end-point case $\sigma =0$, we have to employ the interior Lipschitz estimate under the Dini-type condition (\ref{ineq_Dini_Intro}). Actually, this has been shown in \cite{ShenZhuge2} and here we will give a slightly different approach, based on our new version of Lipschitz estimate with $\lambda = 1$, to obtain the same estimate for $\sigma = 0$.

\begin{theorem}\label{thm_dchiT_S1}
	Suppose that $A\in APW^2(\R^d)$ satisfies the ellipticity condition (\ref{def_ellipticity}) and for some $\sigma \in (0,1)$ and $k\ge 1$, $\omega_{k,\sigma}$ satisfies the Dini-type condition (\ref{ineq_Dini_Intro}). Then,
	\begin{equation}\label{ineq_dchiT_imp}
	\norm{\nabla \chi_T}_{S_1^2} \le C,
	\end{equation}
	where the constant $C$ depends only on $\sigma,k$ and $A$. 
\end{theorem}

\begin{proof}
	Recall that $\chi_{T,j}^\beta$ satisfy the equations for approximate correctors (\ref{def_corrector}). Fix $x_0\in \R^d$, and let
	\begin{equation}
	u_\e(x)= T^{-2}\chi_{T,j}^\beta(Tx) + T P_j^\beta(x-x_0), \qquad T = \e^{-1},
	\end{equation}
	where $P_j^\beta$ is an affine function. Then $u_\e$ satisfies
	\begin{equation}\label{eq_chiT_var}
	-\text{div}(A(x/\e) \nabla u_\e) + u_\e =  \e P_j^\beta(x-x_0), \qquad x\in \R^d.
	\end{equation}
	Therefore, with the additional Dini-type condition on the convergence rate, we can apply the interior Lipschitz estimate to the system (\ref{eq_chiT_var}). It follows that
	\begin{equation*}
	\begin{aligned}
	\left( \fint_{B(x_0,\e)} |\nabla u_\e|^2 \right)^{1/2} & \le C \left( \fint_{B(x_0,1)} |\nabla u_\e|^2 \right)^{1/2} + C \e \left( \fint_{B(x_0,1)} |P_j^\beta(\cdot -x_0)|^p \right)^{1/p}\\
	& \le CT^{-1} \left( \fint_{B(x_0,T)} |\nabla \chi_{T,j}^\beta |^2 \right)^{1/2} + CT^{-1} \\
	& \le CT^{-1},
	\end{aligned}
	\end{equation*}
	where the last inequality follows from (\ref{def_corrector}) and \cite[Lemma 3.1]{ShenZhuge2}. Hence,
	\begin{equation}
	\left( \fint_{B(x_0,1)} |\nabla \chi_T|^2 \right)^{1/2} \le C.
	\end{equation}
	This implies (\ref{ineq_dchiT_imp}) since $x_0\in \R^d$ is arbitrary.
\end{proof}

\subsection{Rellich estimate in $L^2$}
The classical Rellich estimate for harmonic functions claims that $\norm{\nabla u}_{L^p(\partial\Omega)}$, $\norm{\nabla_{\tan} u}_{L^p(\partial\Omega)}$ and $\norm{\frac{\partial}{\partial \nu} u}_{L^p(\partial\Omega)}$ are comparable if they are bounded. These estimates are of importance since they usually imply the solvability of $L^p$ boundary value problems\cite{KS}. In this subsection, we will show the uniform Rellich estimate for $u_\e$ in $L^2$ at large scale in Lipschitz domains without any assumption of smoothness. For simplicity, we temporarily assume $F = 0$ and $\lambda = 0$.

First, we note that if $\Omega$ is a $C^{1,\alpha}$ domain and the conditions of Theorem \ref{thm_Lip_DP} are satisfied, then (\ref{ineq_DP_Lip}) implies
\begin{equation}\label{ineq_Rel_C11}
\left( \fint_{\Omega_r} |\nabla u_\e|^2 \right)^{1/2} \le C \norm{f}_{C^{1,\tau}(\partial\Omega)} ,
\end{equation}
for all $\e\le r\le \text{diam}(\Omega)$, where we also used a covering argument and the energy estimate for $\norm{\nabla u_\e}_{L^2(\Omega)}$. One can see that (\ref{ineq_Rel_C11}) gives a Rellich-type estimate under stronger conditions, which do not imply (\ref{ineq_RelDP}) in Lipschitz domains. However, by taking advantage of some ideas from the proof of uniform Lipschitz estimate, we can easily obtain the following.

\begin{theorem}\label{thm_Rellich_DP}
	Suppose that $A\in APW^2(\R^d)$ satisfies the ellipticity condition (\ref{def_ellipticity}) and $\Omega$ is a bounded Lipschitz domain. Let $u_\e$ be the weak solution of Dirichlet problem
	\begin{equation}
	\cL_\e(u_\e) = 0 \quad \text{in } \Omega, \qquad \text{and} \qquad u_\e =f \quad \text{on } \partial\Omega,
	\end{equation}
	where $f\in H^1(\partial\Omega)$. Then for any $\omega_{k,\sigma}(\e) \le r\le \text{diam}(\Omega)$,
	\begin{equation}\label{ineq_RelD}
	\left( \fint_{\Omega_r} |\nabla u_\e|^2 \right)^{1/2} \le C\norm{\nabla_{\tan} f}_{L^2(\partial\Omega)},
	\end{equation}
	where $C$ is independent of $\e$ and $r$.
\end{theorem}

\begin{proof}
	Recall $\Omega_t = \{x\in \Omega; \text{dist}(x,\partial\Omega) < t \}$. We fix $r>\omega_{k,\sigma}(\e)$ and let 
	\begin{equation}\label{eq_we_4r}
	w_\e = u_\e - u_0 -\e \chi_{T,k}^\beta(x/\e) K_{\e,4r} \left(\frac{\partial u_0^\beta}{\partial x_k}\right).
	\end{equation}
	Now following the same argument of Theorem \ref{thm_H1_Lip}  and choosing $\delta = 4r$ after (\ref{ineq_dKdu0}), we obtain
	\begin{equation}\label{ineq_H1_r}
	\norm{ w_\e}_{H^1(\Omega)} \le Cr^{1/2}  \norm{f}_{H^1(\partial\Omega)}.
	\end{equation}
	Note that this coincides with Theorem \ref{thm_H1_Lip} if $r = \omega_{k,\sigma}(\e)$.
	The point here is that the last term on the right-hand side of (\ref{eq_we_4r}) is supported in $\Omega\setminus\Omega_{2r}$. Thus, by (\ref{ineq_H1_r}) and (\ref{ineq_u0_delta}), we have
	\begin{align*}
	\norm{\nabla u_\e}_{L^2(\Omega_r)} &\le \norm{\nabla w}_{L^2(\Omega)} + \norm{\nabla u_0}_{L^2(\Omega_r)} \\
	&\le Cr^{1/2} \norm{f}_{H^1(\partial\Omega)} \\
	&\le C |\Omega_r|^{1/2}  \norm{f}_{H^1(\partial\Omega)},
	\end{align*}
	where we used the fact $|\Omega_r| \simeq r$ for Lipschitz domains. Finally, note that $u_\e - \int_{\partial\Omega} f$ is also a solution to the same system. Then the last estimate, together with the Poincar\'{e} inequality, gives the desired estimate.
\end{proof}

It is obvious that the proof of Theorem \ref{thm_Rellich_DP} actually has nothing to do with the boundary condition. Therefore, the similar estimate holds for Neumann problem as well.
\begin{theorem}\label{thm_Rellich_NP}
	Suppose that $A\in APW^2(\R^d)$ satisfies the ellipticity condition (\ref{def_ellipticity}) and $\Omega$ is a Lipschitz domain. Let $u_\e$ be the weak solution of Dirichlet problem
	\begin{equation}
	\cL_\e(u_\e) = 0 \quad \text{in } \Omega, \qquad \text{and} \qquad \frac{\partial u_\e}{\partial \nu_\e} =g \quad \text{on } \partial\Omega, \qquad \text{and} \qquad \int_{\Omega} u_\e = 0,
	\end{equation}
	where $F\in L^2(\Omega)$ and $g\in L^2(\partial\Omega)$. Then for any $\omega_{k,\sigma}(\e) \le r \le \text{diam}(\Omega)$,
	\begin{equation}\label{ineq_RelN}
	\left( \fint_{\Omega_r} |\nabla u_\e|^2 \right)^{1/2} \le C\norm{g}_{L^2(\Omega)},
	\end{equation}
	where $C$ is independent of $\e$ and $r$.
\end{theorem}

Strictly speaking, just as the uniform Lipschitz estimate, (\ref{ineq_RelD}) and (\ref{ineq_RelN}) should be called large scale uniform Rellich estimates since the left-hand side is the average integral of $\nabla u_\e$ over a relatively thick boundary layer. To recover the usual Rellich estimates, we must strengthen the conditions from two aspects: 

(1) a better rate of convergence, i.e., $\omega_{k,\sigma}(\e)= O(\e)$ as $\e \to 0$;

(2) symmetry and smoothness conditions on the coefficients, i.e., $A = A^*$ and $A$ is uniformly H\"{o}lder continuous.

With condition (2) above, we are able to bound $\norm{\nabla u_\e}_{L^2(\partial\Omega)}$ by the average integral of $\nabla u_\e$ over the boundary layer $\Omega_{c\e}$. Indeed, it follows from \cite[Theorem 6.3]{KS} and \cite[Remark 3.1]{Shen1} that if $A$ is symmetric and uniformly H\"{o}lder continuous, then
\begin{equation}\label{ineq_RelDs}
\int_{\partial\Omega} |\nabla u_\e|^2 \le C\int_{\partial\Omega} |\nabla_{\tan} u_\e|^2 + \frac{C}{\e} \int_{\Omega_{c\e}} |\nabla u_\e|^2,
\end{equation}
and
\begin{equation}\label{ineq_RelNs}
\int_{\partial\Omega} |\nabla u_\e|^2 \le C\int_{\partial\Omega} \Abs{\frac{\partial u_\e}{\partial \nu_\e}}^2 + \frac{C}{\e} \int_{\Omega_{c\e}} |\nabla u_\e|^2,
\end{equation}
where $C$ and $c$ are independent of $\e$. Now using condition (1) and setting $r = \omega_{k,\sigma}(\e) = C\e $ and combining (\ref{ineq_RelD}), (\ref{ineq_RelN}), (\ref{ineq_RelDs}) and (\ref{ineq_RelNs}), we obtain the usual well-known Rellich estimates:
\begin{equation}
\norm{\nabla u_\e}_{L^2(\partial\Omega)} \le C\norm{\nabla_{\tan} f}_{L^2(\partial\Omega)}, \qquad \norm{\nabla u_\e}_{L^2(\partial\Omega)} \le C\norm{\frac{\partial u_\e}{\partial \nu_\e}}_{L^2(\partial\Omega)}.
\end{equation}

\begin{remark}
	We should mention that the large scale Rellich estimate in $L^p$ can also be established by using the uniform $W^{1,p}$ estimates and convergence rate in $W^{1,p}$ , as shown in \cite{Shen1} (Some conditions of smoothness on $A$ and $\Omega$ are required). However, we will not expand in detail.
\end{remark}

\subsection{Large scale boundary H\"{o}lder estimate} As an easier application of our previous argument for Lipschitz estimate, we will show the uniform H\"{o}lder estimate near the boundary. Let $D_r, \Delta_r$ be defined as before. Let $u_\e \in H^1(D_2;\R^d)$ be a weak solution of $\cL_\e(u_\e) + \lambda u_\e = F$ in $D_2$ with $u_\e = f$ on $\Delta_2$. Here we assume that $F\in L^p(D_2)$ with $p\ge 2$ and $p>d/2$ and $f$ is Lipschitz continuous on $\Delta_2$. Consider the following auxiliary quantity
\begin{equation}\label{def_Phibeta}
\begin{aligned}
\Phi_\gamma (t;u) &= \frac{1}{t^\gamma} \inf_{q\in \R^d} \Bigg\{ \left( \fint_{D_t} |u_\e - q|^2 \right)^{1/2} + t^2 \left( \fint_{D_t} |F|^p \right)^{1/p} \\
&\qquad + t^2\lambda |q| + \norm{f -q}_{L^\infty(\Delta_{t})} + t \norm{\nabla_{\tan} f}_{L^\infty(\Delta_t)}\Bigg\},
\end{aligned}
\end{equation}
where $\gamma < \beta = \min\{2 - d/p,1\}$.

\begin{lemma}\label{lem_uev_holder}
	Let $\e \le r\le 1$. There exists $v\in H^1(D_r;\R^d)$ such that $\cL_0 (v)  + \lambda v = F$ in $D_r$, $v =f$ on $\Delta_r$ and
	\begin{equation}\label{ineq_ueflat_holder}
	\frac{1}{r^\gamma} \left( \fint_{D_r} |u_\e - v|^2\right)^{1/2} \le C[\omega_{k,\sigma}(\e/r)]^{1/2} \Phi_\gamma(2r;u_\e),
	\end{equation}
	where $C$ depends only on $A,\sigma$ and $M$.
\end{lemma}
\begin{proof}
	The proof is exactly the same as Lemma \ref{lem_ue_v}.
\end{proof}

\begin{lemma}\label{lem_Htheta_holder}
	Let $v \in H^1(D_2;\R^d)$ be a weak solution of $\cL_0(v) + \lambda v= F$ in $D_2$ with $ v = f$ on $\Delta_2$. Then there exists $\theta\in (0,1/4)$, depending only on $p,A,\tau,\alpha$ and $M$, such that
	\begin{equation}
	\Phi_\gamma(\theta r; v) \le \frac{1}{2} \Phi_\gamma(r; v).
	\end{equation}
\end{lemma}
\begin{proof}
	The lemma follows from the boundary $C^\alpha$ estimate for the second-order elliptic system with constant coefficients. By rescaling, we can assume that $r = 1$. Let $\gamma < \beta_0 < \beta$ and $q = v(0)$. It is easy to see
	\begin{equation}
	\Phi_\gamma(\theta;v) \le C\theta^{\beta_0 - \gamma} \norm{v}_{C^{\beta_0}(D_\theta)} + C\theta^{2-\gamma - d/p} \left( \fint_{D_1} |F|^p \right)^{1/p}.
	\end{equation}
	Note that $\beta_0 < \beta \le 2-d/p$. Using boundary $C^{\beta_0}$ estimate for $v$, we obtain
	\begin{equation}\label{ineq_vCbeta}
	\norm{v}_{C^{\beta_0}(D_1)} \le C  \Bigg\{ \left( \fint_{D_1} |v|^2 \right)^{1/2} + \left( \fint_{D_1} |F|^p \right)^{1/p} 
	+ \norm{f }_{C^{0,1}(\Delta_1)} \Bigg\}.
	\end{equation}
	Hence,
	\begin{equation}
	\Phi_\gamma(\theta;v) \le C\theta^{\beta_0 - \gamma} \Bigg\{ \left( \fint_{D_1} |v|^2 \right)^{1/2} + \left( \fint_{D_1} |F|^p \right)^{1/p} 
	+ \norm{f }_{C^{0,1}(\Delta_1)} \Bigg\}.
	\end{equation}
	Now let $w = v - q$ for some $q\in \R^d$. Then
	\begin{equation}
	\cL_0(w) + \lambda w = F - \lambda q.
	\end{equation}
	Applying (\ref{ineq_vCbeta}) to $w$ with Dirichlet boundary data $w = f -q$, we arrive at
	\begin{align}
	\begin{aligned}\label{ineq_Phi_w}
	\Phi_\gamma(\theta;w)  &\le C \theta^{\beta_0 - \gamma} \Bigg\{ \left( \fint_{D_1} |v-q|^2 \right)^{1/2}  \\
	&\qquad \qquad \qquad + \left( \fint_{D_1} |F|^p \right)^{1/p} + \lambda |q|
	+ \norm{f-q }_{C^{0,1}(\Delta_1)} \Bigg\}.
	\end{aligned}
	\end{align}
	Observe that
	\begin{equation}\label{ineq_Phi_wv}
	\Phi_\gamma(\theta;v) \le \Phi_\gamma(\theta;w) + \lambda \theta^{2-\gamma} |q|.
	\end{equation}
	Combining (\ref{ineq_Phi_w}) and (\ref{ineq_Phi_wv}) and taking the infimum over all $q\in \R^d$, we obtain
	\begin{equation}
	\Phi_\gamma(\theta;v) \le C\theta^{\beta_0 -\gamma } \Phi_\gamma(1;v).
	\end{equation}
	The desired estimate follows by choosing $\theta \in (0,1/4)$ so small that $C \theta^{\beta_0 -\gamma } \le 1/2$. 
	
\end{proof}

\begin{proof}[Proof of Theorem \ref{thm_holder_DP}]
	By using Lemma \ref{lem_uev_holder}, \ref{lem_Htheta_holder} and the same argument of Lemma \ref{lem_HPhi}, we have
	\begin{equation*}
	\Phi_\gamma(\theta r; u_\e) \le \frac{1}{2} \Phi_\gamma(2r;u_\e) + C [\omega_{k,\sigma}(\e/r)]^{1/2} \Phi_\gamma(2r;u_\e).
	\end{equation*}
	Since $\omega_{k,\sigma}(r) \to 0$ as $r\to 0$, we can choose a particular $N$ sufficiently large such that for any $K\ge N$, we have $C [\omega_{k,\sigma}(K^{-1})]^{1/2} < 1/2$. In other words, for all $N\e \le r< 1/2$, $\Phi_\gamma(\theta r; u_\e) \le \Phi_\gamma(2r; u_\e).$ It follows by iteration that $\Phi_\gamma(r; u_\e) \le \Phi_\gamma(1; u_\e)$ for all $N\e \le r<1/2$.
	Finally the case $\e < r\le N\e$ follows trivially from $\Phi_\gamma(r; u_\e) \le C\Phi_\gamma(N\e; u_\e)$. As a result,
	\begin{equation}
	\Phi_\gamma(r; u_\e) \le C\Phi_\gamma(1; u_\e),
	\end{equation}
	for all $\e < r < 1/2$.
	
	Now by Caccioppoli's inequality,
	\begin{equation*}
	\begin{aligned}
	\left( \fint_{D_{r/2}} |\nabla u_\e|^2 \right)^{1/2} &\le \frac{C}{r} \inf_{q\in \R^d} \Bigg\{ \left( \fint_{D_r} |u_\e - q|^2 \right)^{1/2} + r^2 \left( \fint_{D_r} |F|^p \right)^{1/p} \\
	&\qquad + \norm{f -q}_{L^\infty(\Delta_{r})} + r \norm{\nabla_{\tan} f}_{L^\infty(\Delta_r)}\Bigg\} \\
	& \le C r^{\gamma-1} \Phi_\gamma(r;u_\e) \\
	& \le C r^{\gamma-1} \Phi_\gamma(1;u_\e) \\
	& \le  C r^{\gamma-1} \Bigg\{ \left( \fint_{D_1} |\nabla u_\e |^2 \right)^{1/2} + \left( \fint_{D_1} |F|^p \right)^{1/p} + \norm{f}_{C^{0,1}(\Delta_{1})} \Bigg\}, \\
	\end{aligned}
	\end{equation*}
	which ends the proof.
\end{proof}

The analogous result holds for Neumann problems. By setting
\begin{equation}\label{def_Psibeta}
\begin{aligned}
\Psi_\gamma (t;u) &= \frac{1}{t^\gamma} \inf_{q\in \R^d} \Bigg\{ \left( \fint_{D_t} |u_\e - q|^2 \right)^{1/2} + t^2 \left( \fint_{D_t} |F|^p \right)^{1/p} \\
&\qquad + t^2\lambda |q| + t \norm{g}_{L^\infty(\Delta_t)}\Bigg\},
\end{aligned}
\end{equation}
and applying the similar argument as Theorem \ref{thm_holder_DP}, we have the following.

\begin{theorem}[Boundary H\"{o}lder estimate for NP]\label{thm_holder_NP}
	Suppose that $A\in APW^2(\R^d)$ satisfies the ellipticity condition (\ref{def_ellipticity}). Let $u_\e \in H^1(D_2;\R^d)$ be a weak solution of $\cL_\e(u_\e) + \lambda u_\e= F$ in $D_2$ with $\partial u_\e/\partial \nu_\e = g$ on $\Delta_2$, where $\lambda \in [0,1]$. Then, for any $\e \le r\le 1$,
	\begin{equation}\label{ineq_Holder_NP}
	\left( \fint_{D_r} |\nabla u_\e|^2 \right)^{1/2} \le  C r^{\gamma-1} \Bigg\{ \left( \fint_{D_1} |\nabla u_\e |^2 \right)^{1/2} + \left( \fint_{D_1} |F|^p \right)^{1/p}
	+ \norm{g}_{L^\infty(\Delta_{1})} \Bigg\},
	\end{equation}
	where $\gamma < 2-d/p, p\ge 2, p>d/2$. In particular, if $p = d$, then (\ref{ineq_Holder_NP}) holds for all $\gamma \in (0,1)$.
\end{theorem}

\bibliographystyle{amsplain}
\bibliography{mybib}

\end{document}